\documentclass[11pt]{article}

\usepackage{algorithm}
\usepackage{algpseudocode}
\usepackage{natbib}
\usepackage{amssymb, blkarray, amsmath,xcolor,graphicx,xspace,colortbl,rotating} %
\usepackage{textcomp}
\usepackage{appendix}  
\usepackage{bm}  
\usepackage{boxedminipage}  
\usepackage{color}  
\usepackage{endnotes}  
\usepackage{ragged2e}  
\usepackage[onehalfspacing]{setspace}  
\usepackage{tabulary}  
\usepackage{varioref}  
\usepackage{wrapfig}  
\usepackage{subcaption}

\usepackage[margin=1in]{geometry}
\DeclareGraphicsExtensions {.pdf,.svg,.eps,.ps,.png,.jpg,.jpeg}
\newtheorem {theorem}{Theorem}

\newtheorem {claim}{Claim}

\newtheorem {corollary}{Corollary}

\newtheorem {definition}{Definition}
\newtheorem {example}{Example}

\newtheorem {proposition}{Proposition}
\newtheorem {remark}{Remark}

\newenvironment {proof}[1][Proof]{\noindent \textbf {#1.} }{\ \rule {0.5em}{0.5em}}

\begin{document}

\title{Invariant Distributions in Nonlinear Markov Chains with Aggregators: Theory, Computation, and Applications}

\author{Bar Light\protect\thanks{Business School and Institute of Operations Research and Analytics, National University of Singapore, Singapore. e-mail: \textsf{barlight@nus.edu.sg} }} 
\maketitle

\thispagestyle{empty}

 \noindent \textsc{Abstract}:
\begin{quote}
	
We study the properties of a subclass of stochastic processes called discrete-time nonlinear Markov chains with an aggregator, which naturally appear in various topics such as strategic queueing systems, inventory dynamics, opinion dynamics, and wealth dynamics. In these chains, the next period's distribution depends on both the current state and a real-valued function of the current distribution. For these chains, we provide conditions for the uniqueness of an invariant distribution that do not rely on typical contraction arguments. Instead, our approach leverages flexible monotonicity properties imposed on the nonlinear Markov kernel. We demonstrate the necessity of these monotonicity conditions for proving the uniqueness of an invariant distribution through simple examples. We also provide existence results and introduce an iterative computational method that solves a simpler, tractable subproblem in each iteration and converges to the invariant distribution of the nonlinear Markov chain, even in cases where uniqueness does not hold. We leverage our findings to analyze invariant distributions in strategic queueing systems, study inventory dynamics when retailers optimize pricing and inventory decisions, establish conditions ensuring the uniqueness of solutions for a class of nonlinear equations in 
$\mathbb{R}^{n}$, and investigate the properties of stationary wealth distributions in large dynamic economies.

\end{quote}

\newpage

\section{Introduction}

Nonlinear Markov chains are stochastic processes in which the distribution of the process in the next period  depends on both the current state of the chain and the current distribution. These processes naturally model interacting particle systems and have been extensively studied across various topics in operations, economics, and applied probability, including mean-field games \citep{huang2006large,lasry2007mean,adlakha2013mean}, queueing systems \citep{xu2013supermarket,honnappa2015strategic}, population games \citep{sandholm2010population},  dynamic auctions \citep{iyer2014mean}, nonlinear Monte Carlo algorithms \citep{del2011nonlinear}, stochastic optimization \citep{huaccelerating}, wealth distribution analysis \citep{benhabib2014wealth,ma2020income}, and evolutionary biology \citep{kolokoltsov2010nonlinear}. 

Nonlinear Markov chains with an aggregator are a subclass of nonlinear Markov chains,  where the next period's distribution of the process depends on both the current state of the chain and a real-valued function of the current distribution that is called an aggregator.\footnote{The terminology of `aggregator' originates from the game theory and economics literature, where the process's distribution often represents the distribution of players' states, and the aggregator typically corresponds to a summary statistic such as the mean or a price determined by the entire distribution \citep{acemoglu2012, acemoglu2024equilibrium, light2022mean}. While this paper studies general nonlinear Markov chains that may not necessarily arise from game theory contexts, we still adopt this terminology. Numerous dynamic economic models incorporate an aggregator function, as described in the papers cited above. Nonlinear Markov chains equipped with an aggregator, studied in this paper, capture the dynamics of these systems.} These chains naturally arise in various settings within operations and beyond. For instance, in inventory systems, the aggregator summarizes inventory levels across retailers, influencing replenishment and pricing decisions. In queueing systems, the aggregator can represent expected waiting times, affecting customer arrival behavior. In large dynamic economies, such as those modeling wealth distribution in heterogeneous-agent settings \citep{aiyagari1994} or industry dynamics \citep{weintraub2008markov}, aggregators represent key economic variables like interest rates or equilibrium prices. Additionally, nonlinear Markov chains with an aggregator appear in models of opinion dynamics and other stochastic processes described in \cite{kolokoltsov2010nonlinear}.

The invariant distribution plays an important role in these models as we discuss in our applications. For instance, in dynamic economic models, the invariant distribution  corresponds to the equilibrium of the economy  (see Section \ref{sec:wealth} for a specific example). Similarly, in queueing systems, the invariant distribution describes the stationary distribution of system states, such as queue lengths, which are typically used for  analysis and operational decision-making. Furthermore, in certain other settings, the invariant distribution corresponds to a solution to nonlinear equations. Thus, establishing conditions for uniqueness ensures that these systems yield a single equilibrium or stationary outcome, enabling robust comparative statics across these applications. 

In this paper, we study discrete-time nonlinear Markov chains with an aggregator and provide conditions that ensure the uniqueness of an invariant distribution for these chains without relying on contraction arguments. 
Our approach to prove uniqueness is based on monotonicity properties imposed on the nonlinear Markov kernel. These monotonicity conditions are flexible and can be tailored to the specific nonlinear Markov chain being studied (see the examples in Section \ref{sec:flexibility}). We provide simple examples that demonstrate that uniqueness may fail when the monotonicity conditions do not hold (see Examples \ref{example:decreasing} and \ref{example:preserving} in Section \ref{sec:necessity}).

Additionally, we establish the existence of an invariant distribution under continuity and boundedness assumptions (see Section \ref{Section: Existence}) and introduce a novel algorithm to compute an invariant distribution (see Section \ref{sec:compute}). Crucially, our algorithm does not rely on contraction conditions and finds an invariant distribution  even when uniqueness does not necessarily hold.

In Section \ref{sec:applications}, we explore four distinct applications where our results can be naturally applied. The first application addresses a strategic G/G/1 queueing system, where customer arrivals are influenced by expected waiting times. Under natural conditions on the arrival process that imply that when the expected waiting time is higher fewer agents join the queue,  we demonstrate that there is a unique invariant distribution for the nonlinear dynamics describing the queueing system. We also compute the unique equilibrium expected waiting time for a specific M/G/1 queueing system case. The second application investigates dynamic pricing and inventory replenishment in a revenue management context. Here, a population of ex-ante identical retailers dynamically optimizes pricing and inventory decisions in the face of stochastic demand. Using our framework, we provide an algorithm to find the equilibrium of the system that corresponds to the stationary distribution and establish conditions on inventory dynamics that ensure the existence of a unique invariant distribution, enabling robust comparative statics analysis. The third application studies nonlinear equations in $\mathbb{R}^{n}$, which, despite lacking contraction properties, still possess a unique solution under certain monotonicity conditions that we provide. The fourth application examines the general evolution of wealth distributions within dynamic economic models. We introduce economic assumptions on agents' decisions that ensure the uniqueness of the invariant equilibrium wealth distribution. These applications demonstrate the versatility of our analysis in establishing the uniqueness of an invariant distribution across a diverse set of settings.

\cite{butkovsky2014ergodic} provides conditions for the ergodicity of nonlinear Markov chains. \cite{saburov2016ergodicity} establishes ergodicity conditions for finite state nonlinear Markov chains and \cite{shchegolev2022new} provides improved convergence rates (see \cite{budhiraja2015local} and \cite{ying2018approximation} for further related results). 
However, these approaches fundamentally rely on establishing that the nonlinear Markov operator has contraction properties. This requirement, while leading to strong results like uniform ergodicity, is significantly more restrictive than the conditions for ergodicity in standard linear Markov chains. Crucially, these contraction properties are not satisfied by the models we study in this paper. We discuss this in detail in Section \ref{sec:Non-contraction} in the appendix and show that the nonlinear operators corresponding to our key applications are generally not contractions, even in simple parameter settings. Therefore, these results from the literature are not applicable in our setting.
Additionally, in Example \ref{Example:convergence} in Section \ref{sec:convergence}, we demonstrate that even for one of the most basic nonlinear Markov chains with two states, which satisfies our uniqueness conditions, the chain is not ergodic and 
 does not converge to the unique invariant distribution. This example illustrates that the concepts of uniqueness and ergodicity are distinct, with the separation, intuitively, being more pronounced in nonlinear Markov chains.  In Example \ref{Example: Average Convergence} in Section \ref{sec:convergence} we further show that a law of large numbers does not hold for the nonlinear Markov chains even when our uniqueness conditions hold.  Despite these negative results, we   provide some important applications where the uniqueness of the invariant measure is of interest. For example, the invariant measure can correspond to the solution of nonlinear equations in $\mathbb{R}^{n}$ or the equilibrium wealth distribution of large dynamic economies (see Section \ref{sec:applications}).\footnote{Another  related area of literature is mean field games, where conditions for uniqueness have been studied in \cite{lasry2007mean}, \cite{light2022mean}, and \cite{anahtarci2023learning} in different settings (see also \cite{wikecek2020discrete} and references therein for insights into the connection between discrete-time mean field games and nonlinear Markov chains).} 
  In a continuous time setting with a finite state space, \cite{neumann2023nonlinear} provides conditions that imply the uniqueness of an invariant measure, based on specific assumptions about differentiability and non-singularity related to the generator of the Markov chain. Furthermore,  \cite{neumann2023nonlinear} illustrates peculiar behaviors exhibited by nonlinear Markov chains in continuous time through several examples. Unlike prior works that depend heavily on differentiability or contraction conditions, our results focus on nonlinear Markov chains with an aggregator structure and leverage monotonicity conditions instead. This approach enables us to apply our uniqueness result in settings that previous methods cannot address, such as the applications in Section \ref{sec:applications} we described above. Hence, beyond its theoretical contributions, our results can be used to study invariant distributions and equilibria in practical settings.

\section{Model and Definitions}

This section introduces the model and preliminaries.

\subsection{Nonlinear Markov Chains with an Aggregator} \label{sec:setting}

Let $S$ be a Polish space and $\mathcal{B}(S)$ be the Borel $\sigma$-algebra on $S$. We denote by $\mathcal{P}(S)$ the space of all probability measures on the measurable space $(S,\mathcal{B}(S))$. We study 
 the properties of the nonlinear Markov chain $(X_{t})_{t \in \mathbb{N}}$ on $S$  given by
\begin{equation} \label{Eq:w}
X_{t+1} = w(X_{t},H(\mu _{t}), \epsilon _{t+1})
\end{equation}
where $w: S \times \mathcal{H} \times E \rightarrow S$ is a measurable function, $\mu _{t}$ is the law of $X_{t}$, $H: \mathcal{P}(S) \rightarrow \mathbb{R}$ is a measurable function that is called an aggregator, $\mathcal{H} = \{ H(\mu): \mu \in \mathcal{P}(S) \} $ is the image of $H$, and $(\epsilon_{t})_{t \in \mathbb{N}}$ are independent and identically distributed (i.i.d.) random variables that take values in a Polish space $E$ with a law $\phi$.   We denote and refer to \(h = H(\mu)\) as the aggregator value.

Let $Q$ be the nonlinear Markov kernel that describes the transitions of the nonlinear Markov chain  $(X_{t})_{t \in \mathbb{N}}$, i.e., 
\begin{equation}
    Q(x,h,B) = \phi (\epsilon \in E : w(x,h,\epsilon) \in B)
\end{equation}
for all $B \in \mathcal{B}(S)$, $x \in S$, $h \in \mathcal{H}$. That is, $Q(x,h,B)$ is the probability that the next period's state would lie in the set $B$ when the current state is $x$ and the current aggregator value is $h$. We define the operator \( T: \mathcal{P}(S) \rightarrow \mathcal{P}(S) \) by
$$
T\mu(B) = \int_{S} Q(x, H(\mu), B) \, \mu(dx)
$$
for every measurable set \( B \in \mathcal{B}(S) \). A probability measure \( \mu \in \mathcal{P}(S) \) is an invariant distribution of \( Q \) if it satisfies \( T\mu = \mu \), meaning that \( \mu \) is a fixed point of the operator \( T \).

We are interested in finding conditions that imply that $T$ has a unique fixed point. The operator $T$ is nonlinear and generally not a contraction so standard methods cannot be applied. Instead, we prove uniqueness by leveraging monotonicity conditions over the nonlinear Markov kernel $Q$ that we describe in Section \ref{sec:definition}.

\subsection{Preliminaries} \label{sec:definition} 

We assume throughout the paper that $S$ is endowed with a closed partial order $ \geq $.\footnote{The partial order $ \geq $ on $S$ is closed if $x_{n} \geq y_{n}$ for all $n$, $y_{n},x_{n} \in S$, $y_{n} \rightarrow y$ and $x_{n} \rightarrow x$, $y,x \in S$, imply $x \geq y$. For example, the standard product order on $S \subseteq \mathbb{R}^{n}$ is closed. } 
 We say that a function $f :S \rightarrow \mathbb{R}$ is increasing if $f (y) \geq f (x)$ whenever $y \geq x$. When $S \subseteq \mathbb{R}^{n}$ we will assume that $S$ is endowed with the standard product order unless otherwise stated (that is, $x \geq y$ for $x,y \in \mathbb{R}^{n}$ if $x_{i} \geq y_{i}$ for each $i=1,\ldots,n$). 

The space of probability measures $\mathcal{P} (S)$ is endowed with the weak topology. A sequence of measures $\mu_{n} \in \mathcal{P} (S)$ converges weakly to $\mu \in \mathcal{P} (S)$ if for all bounded and continuous functions $f :S \rightarrow \mathbb{R}$ we have
\begin{equation*}\underset{n \rightarrow \infty }{\lim }\int _{S}f (s) \mu_{n} (d s) =\int _{S}f (s) \mu (d s).
\end{equation*}

To prove that $T$ has a unique fixed point it is  convenient  to assume that the linear Markov kernel  $Q(x,h,\cdot)$ has a unique invariant distribution when the aggregator value $h \in \mathcal{H}$ is fixed. That is, the operator $M_{h} :\mathcal{P} (S) \rightarrow \mathcal{P} (S)$ has a unique fixed point where $M_{h}$ is the operator given by 
\begin{equation*}M_{h} \theta  (B) =\int _{S}Q (x ,h ,B) \theta  (d x)
\end{equation*}
that is parameterized by a fixed aggregator value $h \in \mathcal{H}$.

\begin{definition} \label{def:propU} (Property (U)). 
We say that $Q$ satisfies Property (U) if for any $h \in \mathcal{H}$, the operator $M_{h}$ has a unique fixed point $\mu _{h}$. 
\end{definition}

A stronger version of Property (U), which we refer to as Property (C), states that the Markov kernel $M_{h}^{n} \theta $ converges weakly to $\mu _{h}$ for any probability measure $\theta  \in \mathcal{P} (S)$ where $M_{h}^{n}$ means applying the operator $M_{h}$, $n$ times.

\begin{definition} \label{def:Xerg} (Property (C)). 
We say that $Q$ satisfies Property (C) 
if $Q$ satisfies Property (U) and  $M_{h}^{n} \theta $ converges weakly to $\mu _{h}$ for any probability measure $\theta  \in \mathcal{P} (S)$ and any $h \in \mathcal{H}$ where $\mu _{h}$ is the unique fixed point of $M_{h}$. 
\end{definition}
Under certain conditions, Property (C) can be established using standard results regarding the stability of Markov chains in general state spaces (e.g., Theorem 13.3.1 or Theorem 16.2.3 in  \cite{meyn2012markov}). When the state space $S$ is finite, Property (C) can be established by assuming that $M_{h}$ is irreducible and aperiodic and Property (U) can be established by assuming that $M_{h}$ is irreducible.  

 Let  $D \subseteq \mathbb{R}^{S}$ be a convex set where $\mathbb{R}^{S}$ is the set of all functions from $S$ to $\mathbb{R}$. When $\mu _{1}$ and $\mu _{2}$ are probability measures on $(S ,\mathcal{B}(S))$, we write $\mu _{2} \succeq _{D}\mu _{1}$ if \begin{equation*}\int _{S}f(s)\mu _{2}(ds) \geq \int _{S}f(s)\mu _{1}(ds)
\end{equation*}for all Borel measurable functions $f \in D$ such that the integrals exist. With slight abuse of notation, for two random variables $X,Y$, we write $X \succeq_{D} Y$ if $\mu _{X} \succeq _{D} \mu_{Y} $ where $ \mu_{X}$ is the law of $X$ and $\mu_{Y}$ is the law of $Y$.  
The binary relation $\succeq _{D}$ is called a stochastic order. 

The key assumption that implies that the operator $T$ has at most one fixed point relates to the following monotonicity and preservation properties. 
\begin{definition}
Let $D \subseteq \mathbb{R}^{S}$. 

We say that $Q$ is $D$-decreasing  if for each $x \in S$, we have $Q (x ,h_{1} , \cdot ) \succeq  _{D}Q (x , h_{2}, \cdot )$ whenever $h_{2} \geq h_{1}$, $h_{1},h_{2} \in \mathcal{H}$.

We say that  $Q$ is $D$-increasing in $x$  with respect to $\succeq _{D}$ if for each  $h \in \mathcal{H}$, we have $Q (x _{2},h , \cdot ) \succeq  _{D}Q (x_{1} ,h, \cdot )$ whenever $x_{2} \geq x_{1}$.

We say that $Q$ is $D$-preserving if for all  $h \in \mathcal{H}$ the function 
\begin{equation*}
    v(x):=\int f (y) Q(x ,h ,dy) 
\end{equation*}
is in $D$ whenever $f \in D$.

\end{definition}

Note that when \(D\) is the set of all increasing functions, the order \(\succeq_D\) coincides with standard stochastic dominance and we write $\mu _{2} \succeq _{SD}\mu _{1}$ and say that $\mu _{2}$ first order stochastically dominates $\mu _{1}$. In this case, \(Q\) is \(D\)-increasing in \(x\) with respect to \(\succeq_D\) if and only if it is \(D\)-preserving (see, for example, Corollary~3.9.1 in \cite{topkis2011supermodularity}). 
Intuitively, when \(D\) consists of increasing functions, the \(D\)-preserving property means that, for any fixed aggregator value \(h\), the Markov kernel \(Q(\cdot,h,\cdot)\) is monotone in the sense of stochastic dominance. This property is well studied in the literature on monotone Markov chains and is often straightforward to verify in applications such as those in Section~\ref{sec:applications}. In this setting, the \(D\)-decreasing condition implies that larger aggregator values shift the distribution downward in the stochastic order sense, i.e., higher aggregator values reduce the likelihood of larger future states.
The same intuition extends to other structural classes beyond monotonicity. For instance, in models with complementarities (such as Example~\ref{example:congestion} in Section \ref{sec:flexibility}), one may take \(D\) to be the set of supermodular functions. In this case, \(D\)-preserving means that the Markov kernel with fixed aggregator value \(h\) preserves the complementarity structure over time.

We also note that when \(Q\) is \(D\)-increasing in \(x\) and $D$ is the set of all increasing functions, Property~(C) can often be verified using results from the theory of monotone Markov chains. These results typically rely on a splitting condition (see \cite{bhattacharya1988asymptotics}, \cite{kamihigashi2014stochastic}, and \cite{light2024note}) and apply to a wide range of models in operations. 

We say that $H$ is increasing with respect to $\succeq _{D}$ if $H(\mu_{2}) \geq H(\mu_{1})$ whenever $ \mu_{2} \succeq _{D} \mu_{1}$.

A stochastic order $\succeq_{D}$ is said to be closed with respect to weak convergence if $\mu^{1}_{n} \succeq_{D} \mu^{2}_{n} $ for all $n$, $\mu^{1}_{n}$ converges weakly to $\mu^{1}$, and $\mu^{2}_{n}$ converges weakly to $\mu^{2}$ imply $\mu^{1} \succeq_{D} \mu^{2}$. Many stochastic orders of interest are closed with respect to weak convergence, e.g., the standard stochastic dominance order $\succeq_{SD}$. For a textbook treatment  of the closure properties of stochastic orders see, for example, Theorems 4.B.10  and 3.A.5  in \cite{shaked2007stochastic} .   

We say that $H$ is continuous if  $ \lim _{n \rightarrow \infty} H(\mu_{n}) = H(\mu)$ whenever $\mu_{n}$ converges weakly to $\mu$. We say that $Q$ is continuous if $Q(x_{n},h_{n}, \cdot)$ converges weakly to  $Q(x,h, \cdot )$ whenever $(x_{n},h_{n}) \rightarrow (x,h)$. Also, for a parametrized random variable \( Y(z) \) depending on a parameter $z \in \mathbb{R}^{n}$, we say that \( Y(z) \) is continuous in \( z \) if \( z_n \to z \) implies that the law of \( Y(z_n) \) converges weakly to the law of \( Y(z) \).

Recall that a partially ordered set $(Z , \geq )$ is said to be a lattice if for all $x ,y \in Z$, $\sup \{x ,y\}$ and $\inf \{y ,x\}$ exist in $Z$. $(Z , \geq )$ is a complete lattice if for all non-empty subsets $Z^{ \prime } \subseteq Z$ the elements $\sup Z^{ \prime }$ and $\inf Z^{ \prime }$ exist in $Z$.

\section{Main Results}

In this section we present our main results.  In Section \ref{sec:unique} we present the monotonicity conditions that imply that  the nonlinear Markov chain has at most one invariant distribution. In Section \ref{Section: Existence} we provide two distinct existence results. In Section \ref{sec:flexibility} we provide examples that demonstrate the flexibility of the monotonicity conditions. In Section \ref{sec:necessity} we show that these monotonicity conditions are necessary to prove uniqueness in our setting and in Section \ref{sec:convergence} we show that the nonlinear Markov chain does not necessarily converge to the unique invariant distribution even for a very simple two-state case. In Section \ref{sec:compute} we provide a simple method to compute the invariant distribution. In Section \ref{sec:localresults} we provide local uniqueness results.

\subsection{Uniqueness Theorem} \label{sec:unique}
In this section we provide the monotonicity conditions that ensure \( Q \) has at most one invariant distribution. 
The proofs of all the paper's results are deferred to the Appendix.

\begin{theorem} \label{Theorem: unique}
Let $D \subseteq \mathbb{R}^{S}$ be a non-empty set such that $H$ is increasing with respect to $\succeq  _{D}$. Assume that $Q$ is $D$-preserving and $D$-decreasing.

Assume that either of the following conditions hold: 

(i) $Q$ satisfies Property (C) and $\succeq_{D}$ is closed with respect to weak convergence. 

(ii) $Q$ satisfies Property (U) and $(\mathcal{P}(S),\succeq_{D})$ is a complete lattice.

Then $Q$ has at most one invariant distribution. 
\end{theorem}

The conditions in Theorem \ref{Theorem: unique}, which establish that \( Q \) has at most one invariant distribution, do not rely on compactness or continuity assumptions, and hence, the existence of an invariant distribution is not guaranteed.
In Section \ref{Section: Existence}, we present conditions that ensure the existence of an invariant distribution.

We now provide a few comments on Theorem \ref{Theorem: unique}.

\textbf{Applications:} In many applications, verifying whether the nonlinear Markov kernel \( Q \) is both \( D \)-preserving and \( D \)-decreasing is straightforward. In Section \ref{sec:applications}, we present several applications of Theorem \ref{Theorem: unique}, including queueing systems and the dynamic evolution of wealth distributions. In these cases, the monotonicity properties of \( Q \) naturally arise from the underlying behavioral or economic assumptions governing the dynamics of the stochastic systems.

\textbf{Non-contraction:} We show in detail in Section~\ref{sec:Non-contraction} that standard contraction-based techniques for proving uniqueness of invariant distributions generally fail for the nonlinear Markov chains studied in this paper, even in simple special cases arising in our applications. In particular, we show that the associated operator \( T \) is not a contraction under either the Wasserstein or total variation metrics. Hence, standard contraction methods for proving uniqueness are not applicable in our setting. We show in Section \ref{sec:Non-contraction} that these failures are representative across all our examples and applications, rather than peculiar to specific constructions.\footnote{As an alternative, we also discuss in Section \ref{sec:Non-contraction} an augmented formulation in which the nonlinear Markov chain is embedded into a higher-dimensional linear Markov chain. However, we show that even this chain does not satisfy the conditions required for standard uniqueness results.}  This further motivates our monotonicity-based approach.

\textbf{Local results:} The proof of Theorem \ref{Theorem: unique} indicates that it suffices to assume Property (U) only for \( h \in \mathcal{H} \), where \( h = H(\mu) \) and \( \mu \) is an invariant distribution of \( Q \). This relaxation means that Property (U) does not need to hold for all \( h \in \mathcal{H} \), which can simplify the verification of the condition in specific applications.

The monotonicity conditions required for proving Theorem \ref{Theorem: unique} are global, meaning they must hold across all probability measures on $S$. However, in certain applications, only a subset of these probability measures includes relevant candidates for invariant distributions or is of particular interest. In Proposition \ref{Prop:local}, introduced in Section \ref{sec:localresults}, we provide a local version of Theorem \ref{Theorem: unique} that allows for establishing uniqueness within a restricted set of probability measures.

\textbf{The finite case and complete lattices:}  Condition (ii) of Theorem \ref{Theorem: unique} is particularly useful for the case that $S$ is a finite set or a compact set in $\mathbb{R}$.
For example, suppose that $S=\{s_{1},\ldots,s_{n}\}$ is an ordered set of numbers with $s_{1} \leq s_{2} \leq ... \leq s_{n}$ and $\mathcal{P}(S)$ is endowed with the standard stochastic dominance order $\succeq  _{SD}$. It is immediate to see that $(\mathcal{P}(S),\succeq_{SD})$ is a complete lattice with
$$ \sup \{ \mu , \lambda \}  (\{s_{t}, \ldots, s_{n} \})=  \max \{ \mu   (\{s_{t}, \ldots, s_{n} \}) ,\lambda   (\{s_{t}, \ldots, s_{n} \}) \} $$
and 
$$ \inf \{ \mu , \lambda \} (\{s_{t}, \ldots, s_{n} \})=  \min \{ \mu (\{s_{t}, \ldots, s_{n} \}) ,\lambda  (\{s_{t}, \ldots, s_{n} \}) \} $$
for all $t =1,\ldots, n$  
(recall that $\mu  \succeq  _{SD} \lambda$ if and only if for every upper set $B$ we have $\mu (B) \geq \lambda (B)$ where $B \in \mathcal{B} (S)$ is an upper set if $x_{1} \in B$ and $x_{2} \geq x_{1}$ imply $x_{2} \in B$). In a similar fashion, $(\mathcal{P}(S),\succeq_{SD})$ is a complete lattice  when $S$ is a compact set in $\mathbb{R}$ when $\mathbb{R}$ is endowed with the standard partial order. For this result and other examples of stochastic orders that generate lattices of probability measures see  \cite{muller2006stochastic}.

\subsection{Existence of Invariant Distribution} \label{Section: Existence}
In this section, we study the existence of an invariant distribution. We present two distinct results. 
The first existence result, Proposition \ref{Prop:existence}, holds  for the case where $S$ is compact and $Q$ and $H$ are continuous and follows from standard fixed-point arguments. Extending this existence result to non-compact state spaces remains an interesting avenue for future research.

\begin{proposition} \label{Prop:existence}
Suppose that $H$ and $Q$ are continuous and that $S$ is compact. Then $Q$ has an invariant distribution. 
\end{proposition}

The second existence result relies on continuity of $H$ and $Q$, a boundedness condition for the aggregator and a tightness condition instead of compactness of the state space. This result is especially useful in applications where the state space is not finite or compact, such as the queueing systems studied in Section \ref{sec:applications} or the autoregressive processes discussed in Example \ref{example:flexible}.

Recall that a sequence of probability measures $\{\mu_{k}\}$ on $S$ is called tight if for all $\epsilon > 0$ there is a compact subset $K_{\epsilon}$ of $S$ such that $\mu_{k}(S \setminus K_{\epsilon}) \leq \epsilon$ for all $k$. Tightness is a standard assumption in order to ensure the existence of an invariant distribution in the usual linear Markov chain theory (see \cite{meyn2012markov} for an extensive study of invariant distributions).

\begin{proposition} \label{Prop:existence2}
Suppose that $H$ and $Q$ are continuous and that Property (U) holds. In addition, assume that there exist $h',h'' \in \mathbb{R}$, $h''>h'$, such that $h'' \geq H(\mu_{h''})$ and $h' \leq H(\mu_{h'})$ where $\mu_{h}$ is the unique fixed point of $M_{h}$ (see Definition \ref{def:propU}) and $\mu_{h} \in \mathcal{P}(S)$ for all $h \in [h',h'']$. 
Assume that for any sequence $\{h_{n}\} $, $h_{n} \in [h',h'']$ that converges to some $h$, the sequence $\{\mu_{h_{n}} \}$ is a tight sequence of probability measures. 

Then $Q$ has an invariant distribution.\footnote{\label{footnote:Hcontinuous}From the proof of Proposition \ref{Prop:existence2}, it is immediate that it is enough to require that $H$ is continuous on $\{\mu_{h} \in \mathcal{P}(S) : h \in [h', h'']\}$.
} 
\end{proposition}

The existence result in Proposition \ref{Prop:existence2}
 not only establishes the existence of an invariant distribution but also provides the basis for an algorithm for finding this distribution. Specifically, we provide a bisection method to find the invariant distribution (see Section \ref{sec:compute}) which complements the theoretical existence results.



\subsection{Flexibility of the Monotonicity Conditions} \label{sec:flexibility}

In applications, it is common to select \( D \) as the set of all increasing functions, corresponding to the standard stochastic dominance order. 
However, Theorem~\ref{Theorem: unique} allows for greater flexibility in the choice of \( D \), and this flexibility can be essential in certain settings. 
For instance, when complementarity across states plays a key role as in Example~\ref{example:congestion}, choosing \( D \) as the set of supermodular functions enables proving uniqueness.  
While the examples below are somewhat specialized, they show the importance of tailoring the choice of \( D \) to the structure of the model.

\begin{example} \label{example:flexible} (Flexibility of the set $D$).
 Consider the following nonlinear Markov chain 
\begin{equation} \label{Eq:ex_1(i)} X_{t+1} = a X_{t} - H(\mu_{t}) + \epsilon_{t+1} \end{equation}
on $\mathbb{R}$ where $0<a<1$, $\epsilon _{t}$ are i.i.d. random variables with finite expectations and variances, and  the aggregator is given by $H(\mu) = \int m(x)  \mu(dx)$ for some increasing, continuous function $m:\mathbb{R} \rightarrow \mathbb{R}$ such that $|m(x)| \leq C_{0} + C_{1}|x|^{2}$ for all $x$ and $C_{0}, C_{1} \geq 0$.\footnote{This quadratic growth condition can be replaced by a higher-order polynomial growth of the form $|x|^{k}$, as long as the $k$th moment of $\epsilon_t$ is finite.} Then, we can use Theorem \ref{Theorem: unique}  to show that the nonlinear Markov chain $(X_{t})_{t \in \mathbb{N}}$ has at most one invariant distribution and Proposition \ref{Prop:existence2} to show that an invariant distribution exists.  The proofs of the claims are provided in the appendix. 
\begin{claim} \label{claim1} {The Markov chain given in Equation (\ref{Eq:ex_1(i)}) has a unique invariant distribution.}
\end{claim}

Now consider the nonlinear Markov chain \begin{equation}\label{eq:ex_1_2dim} (X_{1,t+1},X_{2,t+1}) =  (a X_{1,t} - H(\mu_{t}) + \epsilon_{1,t+1}, k(X_{2,t}) + \epsilon_{2,t+1})\end{equation} 
on $\mathbb{R}^{2}$ where $0<a<1$, $\epsilon _{1,t} , \epsilon_{2,t}$ are i.i.d. random variables with finite expectations and variances, $k$ is a function that is continuous and bounded but not increasing, and the aggregator is given by $H(\mu) = \int m(x_{1})  \mu(d(x_{1},x_{2}))$ for some increasing continuous and bounded function $m:\mathbb{R} \rightarrow \mathbb{R}$. In this case,   $Q$ is not necessarily $D$-preserving when $D$ is the set of all increasing functions because $k$ is not increasing. However, if we let $D$ to be the set of all the functions that are increasing in the first argument, it can be verified that $Q$ is both $D$-preserving and $D$-decreasing (see the claim below).  
\begin{claim} \label{claim2} {Consider the Markov chain given in Equation (\ref{eq:ex_1_2dim}). Then it has a unique invariant distribution if Property (C) holds.\footnote{Establishing Property (C) for such Markov chains has been extensively studied in the literature \citep{meyn2012markov} so we omit the details for brevity. }} 
\end{claim}
\end{example}

\begin{example} \label{Example:flexible2}
(Flexibility of the set $D$). Consider the $n$-dimensional nonlinear Markov chain on $\mathbb{R}^{n}$ 
with 
\begin{equation} \label{Eq_Ex1_ii} X_{i,t+1} =  a_{i} X_{i,t} - \beta_{i} H(\mu_{t}) + \epsilon_{i,t+1} \end{equation}
for $i=1,\ldots,n$ where $0<a_{i}<1$, $\epsilon _{i,t} $ are i.i.d. random variables with finite expectations and variances,  and the aggregator is given by 
$H(\mu)  = \int \sum _{i=1}^{n} 
 \gamma_{i} x_{i}\mu(d(x_{1},x_{2},\ldots,x_{n}))$ for some vector $\gamma = (\gamma_{1},\ldots,\gamma_{n})$ in $\mathbb{R}^{n}$. 

 Let $O$ be the set of vectors in $\mathbb{R}^{n}$ such that $x_{i}$ is non-negative for an odd $i$ and non-positive for an even $i$, that is, $O = \{ x \in \mathbb{R}^{n}: x_{i} \geq 0, i \text{ is odd }, x_{i} \leq 0,  i \text{ is even}\} $. Assume that $\beta = (\beta _{1} ,\ldots,\beta_{n})$ and $\gamma = (\gamma_{1} , \ldots ,\gamma_{n})$ are in $O$. It is easy to see that we cannot use $D$ as the set of all increasing functions in order to apply Theorem \ref{Theorem: unique}. However, 
 consider the set of functions $D$ such that $f: \mathbb{R}^{n} \rightarrow \mathbb{R}$ is in $D$ if $f(x) = \sum _{i=1}^{n} y_{i}x_{i} + c$ for some $y \in O$ and $c \in \mathbb{R}$. Under this set of functions $D$, we show that we can use Theorem \ref{Theorem: unique} to prove that the nonlinear Markov chain has at most one invariant distribution. 
 \begin{claim} \label{claim3} The nonlinear Markov chain given in Equation (\ref{Eq_Ex1_ii}) has a unique invariant distribution.
 \end{claim}
\end{example}

We note that contraction properties are generally not satisfied for the nonlinear Markov chains discussed in the preceding examples (see Section~\ref{sec:Non-contraction} in the Appendix for a detailed discussion on this and counterexamples). 

Other important examples of the flexibility of $D$ arise when using the supermodular stochastic order and the convex stochastic order. Supermodularity is a natural structural property in many models in operations and economics (see, e.g., \cite{topkis2011supermodularity}), while the convex order plays a key role in risk analysis through the concept of mean-preserving spreads.
We now present an example where $D$ is taken to be the set of supermodular functions, which enables us to establish uniqueness, whereas choosing $D$ as the set of increasing functions would fail. In Section \ref{Sec:convexOrder} in the appendix, we provide a complementary example where $D$ is the set of convex functions again leading to uniqueness.

\begin{example}\label{example:congestion}
 (Supermodular Stochastic Order). Consider a continuum of users interacting with a shared service (e.g.,  cloud service).  Each user’s state at time \(t\) is described by a binary vector 
$
X_t = (X_{1,t},\dots,X_{n,t}) \in \{0,1\}^n,
$ 
where \(X_{i,t}=1\) indicates that component \(i\) (e.g.\ data loading, model training, evaluation) is running smoothly, and \(X_{i,t}=0\) indicates slowdown or timeout.  Define the system-wide congestion level by
$ H(\mu_t)
     \;=\;\mu_t\bigl(\{(1,\dots,1)\}\bigr)
$ which measures the fraction of users for whom every component of their workflow is running smoothly, and hence, serves as a proxy for the aggregate load placed on the shared infrastructure.

We define the nonlinear Markov chain
\begin{equation}
X_{t+1} \label{Equation:Supermodular}
=\begin{cases}
(Z,\dots,Z), &\text{w.p. }\rho(X_t,h_t),\;Z\sim\mathrm{Bern}(p),\\[4pt]
(Z_1,\dots,Z_n), &\text{w.p. }1-\rho(X_t,h_t),\;Z_i\overset{\mathrm{i.i.d.}}{\sim}\mathrm{Bern}(p).
\end{cases}
\end{equation}
where $p \in (0,1)$ is some baseline probability, $\mathrm{Bern}(p)$ is the Bernoulli random variable that yields $1$ with probability (w.p.) $p$ and zero otherwise, and $\rho (x,h) \in [0,1)$ is a continuous function that is decreasing in $h$ and supermodular in $x$.\footnote{Recall that for a lattice $S$, a function \( f : S \to \mathbb{R} \) is said to be \emph{supermodular} if for all \( x, y \in S \),
\[
f(x) + f(y) \leq f(x \vee y) + f(x \wedge y),
\]
where \( x \vee y \) and \( x \wedge y \) denote the least upper bound (join) and greatest lower bound (meet) of \( x \) and \( y \) in the lattice \( S \), respectively.
} 
  
In these dynamics, \(\rho\) can be seen as the probability of coordination.\footnote{For simplicity of exposition, we choose the perfectly coordinated vector \((Z,\dots,Z)\). However, the uniqueness proof in Claim~\ref{claim:supermodular}, which is based on the supermodular order, does not rely on this specific choice. The argument extends immediately to other coordinated vectors, as long as they are positively associated or comonotone and dominate \((Z_1,\dots,Z_n)\) in the supermodular order (see, e.g., \cite{hu2005dependence} for an establishment of such domination).
 } These dynamics encode two complementary forces. First, as the aggregate congestion level \(h_t\) rises, the shared infrastructure becomes strained, reducing the probability of a  coordinated shock. 
Second, the supermodularity of \(\rho\) in \(x\) implies that adding an additional smoothly running component increases the probability of coordinated continuation more when the other components are already smooth, capturing the natural idea that coordination is more likely when most components are already aligned. 

We now show that by choosing \( D \) as the set of supermodular functions, we can apply Theorem~\ref{Theorem: unique} to establish uniqueness.  In addition, in the proof of Claim~\ref{claim:supermodular} we show that taking \( D \) to be the set of increasing functions fails to satisfy the required monotonicity conditions for uniqueness.

\begin{claim} \label{claim:supermodular}
    The nonlinear Markov chain given in Equation (\ref{Equation:Supermodular}) has a unique invariant distribution. 
\end{claim}

\end{example}

\subsection{Necessity of the Monotonicity Conditions} \label{sec:necessity} 
In this section, we show that without the \( D \)-preserving and \( D \)-decreasing properties, there are simple examples in which uniqueness of the invariant distribution fails.

\begin{example} \label{example:decreasing}  ($Q$ is not $D$-decreasing). Suppose that $S=\{0,1\}$ endowed with the standard order ($1 \geq 1, 0 \geq 0,  1 > 0$) and $H(\mu) = \mu (\{1\}) $. Assume that $D$ is the set of all increasing functions so $\succeq_{D}$ is the standard stochastic dominance $\succeq_{SD}$. Note that  $H$ is increasing with respect $\succeq_{SD}$.   

 Consider the nonlinear Markov chain 
\[
Q'  = 
        \begin{blockarray}{c@{\hspace{2pt}}rr@{\hspace{3pt}}}
         & 0   & 1   \\
        \begin{block}{r@{\hspace{2pt}}|@{\hspace{2pt}}
    |@{\hspace{2pt}}rr@{\hspace{2pt}}||}
        0 & 1- \min(0.5,\mu (\{1\} )) & \min(0.5,\mu (\{1\})) \\
        1 & 0.5 & 0.5  \\
        \end{block}
    \end{blockarray}
\]
It is immediate that $ \pi (\{ 1 \}) = 1/2 = \pi(\{0\})$ and $ \pi ' (\{ 1 \}) = 0,  \pi '  (\{ 0 \}) = 1$ are invariant distributions of $Q'$. 
It is easy to verify that $Q'$ satisfies property (ii) of Theorem 1, and that $Q'$ is $D$-preserving but not $D$-decreasing. Hence all the conditions of Theorem 1 are satisfied except for the condition that $Q'$ is $D$-decreasing and $Q'$ has two invariant distributions. 

\end{example}

\begin{example} \label{example:preserving}
 ($Q$ is not $D$-preserving). Suppose that $S=\{0,1,2\}$ is endowed with the standard order and $H(\mu) = \mu (\{1\}) + \mu (\{2\}) $. Assume that $D$ is the set of all increasing functions so $\succeq_{D}$ is the standard stochastic dominance $\succeq_{SD}$. Note that  $H$ is increasing with respect $\succeq_{SD}$.

Consider the nonlinear Markov chain  
\[
Q'' = 
        \begin{blockarray}{c@{\hspace{1pt}}rrr@{\hspace{3pt}}}
         & 0   & 1   & 2 \\
        \begin{block}{r@{\hspace{1pt}}|@{\hspace{1pt}}
    |@{\hspace{1pt}}rrr@{\hspace{1pt}}|@{\hspace{1pt}}|}
        0 & 1/3 & 1/3 & 1/3 \\
        1 & 0 & H(\mu) & 1 - H(\mu) \\
        2 & H(\mu)  & 0   & 1-H(\mu)   \\
        \end{block}
    \end{blockarray}
\]

The distributions $ \pi (\{ 0 \}) =  \pi (\{ 1 \}) =  \pi (\{ 2 \}) = 1/3$ and $ \pi' (\{ 0 \}) = 0,  \pi' (\{ 1 \}) = 1,  \pi' (\{ 2 \})=0 $  are invariant distributions of $Q''$. 
It is easy to see that the Markov chain $Q''$ satisfies property (ii) of Theorem 1 and is $D$-decreasing. In addition, $Q''$ is not increasing in $x$, and hence, is not $D$-preserving as $Q''(1,h , \{1,2\} ) > Q''(2,h , \{1,2\} )$ for any $h >0$. Hence all the conditions of Theorem 1 are satisfied except for the condition that $Q''$ is $D$-preserving and $Q''$ has two invariant distributions. 

\end{example}


\subsection{Non-Convergence to the Invariant Distribution} \label{sec:convergence} Theorem \ref{Theorem: unique} and Proposition \ref{Prop:existence} provide sufficient conditions for the uniqueness of an invariant distribution for the nonlinear Markov kernel $Q$. However, these results do not provide conditions under which the sequence of measures $\mu_{t}$  converges weakly to the unique invariant distribution of $Q$.  Unfortunately, the following example shows that even in a very simple case, the monotonicity conditions that imply uniqueness do not imply convergence.  This is in sharp contrast with the contraction approach to study the invariant distributions of nonlinear Markov chain that guarantees convergence (e.g., \cite{butkovsky2014ergodic}). In Section~\ref{sec:Non-contraction} we show that the non-contraction property extends to the applications studied in this paper.

\begin{example} \label{Example:convergence} ($\mu_{t}$ does not converge to the unique invariant distribution).
Suppose that $S=\{0,1\}$ is endowed with the standard order and $H(\mu) = \mu (\{1\}) $. Assume that $D$ is the set of all increasing functions so $\succeq_{D}$ is the standard stochastic dominance $\succeq_{SD}$. Note that $H$ is increasing with respect $\succeq_{SD}$.    Consider the nonlinear Markov chain 
\[
Q = 
        \begin{blockarray}{c@{\hspace{2pt}}rr@{\hspace{3pt}}}
         & 0   & 1   \\
        \begin{block}{r@{\hspace{2pt}}|@{\hspace{2pt}}
    |@{\hspace{2pt}}rr@{\hspace{2pt}}||}
        0 & \mu (\{1\} ) & \mu (\{0\}) \\
        1 & \mu (\{1\}) & \mu (\{0\})  \\
        \end{block}
    \end{blockarray}
\]
It is easy to see that $ \pi (\{ 1 \}) = 1/2 = \pi(\{0\})$ is the unique invariant distribution of $Q$ and $Q$ 
 satisfies all the conditions of Theorem 1. Note that for any initial distribution $\mu_{1} (\{1\} )  = \gamma $ and $\mu_{1} (\{ 0 \}) = 1-\gamma$ with $\gamma \neq 1/2$, $\mu_{t}$ does not converge to $\pi$ as $\mu_{t} (\{1\} )  = \gamma $ and $\mu_{t} (\{ 0 \}) = 1-\gamma$ for an odd $t$ and $\mu_{t} (\{1\} )  = 1- \gamma $ and $\mu_{t} (\{ 0 \}) = \gamma$ for an even $t$. 
\end{example}

 Example \ref{Example:convergence} illustrates that the sequence of measures $\{ \mu _{t} \}$ does not converge to the unique invariant distribution in a simple example showing  that we can't expect the sequence of measures $\{ \mu _{t} \}$ to converge in typical applications. In that example, $ \sum _{t=1}^{T} \mu_{t} /T$ converges to the unique invariant distribution. However, Example \ref{Example: Average Convergence} shows that this is not always the case even when the conditions for uniqueness provided in Theorem \ref{Theorem: unique} hold. 

 \begin{example} \label{Example: Average Convergence}
     ($\sum _{t=1}^{T} \mu_{t} /T$ does not converge to the unique invariant distribution).
Suppose that $S=\{0,1\}$ is endowed with the standard order and $H(\mu) = \mu (\{1\}) $. Assume that $D$ is the set of all increasing functions so $\succeq_{D}$ is the standard stochastic dominance $\succeq_{SD}$. Note that $H$ is increasing with respect $\succeq_{SD}$.    Consider the nonlinear Markov chain 
\[
Q = 
        \begin{blockarray}{c@{\hspace{2pt}}rr@{\hspace{3pt}}}
         & 0   & 1   \\
        \begin{block}{r@{\hspace{2pt}}|@{\hspace{2pt}}
    |@{\hspace{2pt}}rr@{\hspace{2pt}}||}
        0 & \mu (\{1\} ) & \mu (\{0\}) \\
        1 & 1-f( \mu (\{0\})) & f( \mu (\{0\}))  \\
        \end{block}
    \end{blockarray}
\]
with 
\begin{align*}
    f(x ) & = x1_{ \{ x \leq 0.3 
\} } + (1.2x  -0.06) 1_{ \{ 0.3 < x \leq  0.5  \} } 
+  (0.8x  + 0.14) 1_{ \{ 0.5 < x \leq  0.7 \}  } +  x1_{ \{x > 0.7 
\} }
\end{align*} 
for $x \in [0,1]$. 
Note that $f (x ) \geq x$ and $f$ is increasing, and hence, the conditions of Theorem \ref{Theorem: unique} hold and there exists at most one invariant distribution. In addition, $f$ is continuous so from Proposition \ref{Prop:existence} the nonlinear Markov kernel $Q$ has a unique invariant distribution.

As in Example \ref{Example:convergence}, if the initial distribution is $\mu _{1} (\{0\} )  = 0.7$, then $\mu _{2}(\{0\})  = 0.3$, and $\mu_{3} (\{0\})  = 0.7$ and so on. 
But $\pi ( \{ 0 \} ) = \pi (\{1\}) = 1/2 $ is not an invariant distribution so $\sum _{t=1}^{T} \mu_{t} /T$ does not converge to the invariant distribution.  

 \end{example}


\subsection{Computation of the  Invariant Distribution } \label{sec:compute}

As discussed in the introduction, it is essential to develop a method capable of computing the invariant distribution of the nonlinear Markov chain. In this section, under the conditions of Proposition \ref{Prop:existence2},  we show that a straightforward bisection method achieves this computational goal.  In this method, we use bisection method for the function $f(h) = h - H(\mu_{h})$ on the interval $[h',h'']$  to find the root of $f$. 
We now describe a simple algorithm to compute the invariant distribution of $Q$.

\begin{algorithm}[htbp]
\caption{Bisection Method for Finding an Invariant Distribution}
\label{alg:bisection_simplified}
\begin{algorithmic}[1]
  \Require Interval \([h',h'']\) with \(f(h')<0<f(h'')\), tolerances \(\varepsilon_x>0\) and \(\varepsilon_y>0\) 
  \Ensure Approximate root \(h^*\) and invariant measure \(\mu_{h^*}\) satisfying \(\lvert b - a\rvert \le \varepsilon_x\) or \(\lvert f(h^*)\rvert \le \varepsilon_y\)
  \State \(a \gets h'\), \(b \gets h''\)
  \Repeat
    \State \(h_n \gets (a + b)/2\) \Comment{Midpoint}
    \State Compute \(\mu_{h_n}\) by solving \(\mu_{h_n}(B) = \int Q(x,h_n,B)\,\mu_{h_n}(dx)\)
    \State \(f(h_{n}) \gets h_n - H(\mu_{h_n})\)
    \If{\(\lvert f(h_{n})\rvert \le \varepsilon_y\)} 
      \State \(h^* \gets h_n\), \(\mu_{h^*} \gets \mu_{h_n}\) 
      \State \textbf{break}
    \ElsIf{\(f(h_{n}) < 0\)}
      \State \(a \gets h_n\)
    \Else
      \State \(b \gets h_n\)
    \EndIf
  \Until{\(\lvert b - a\rvert \le \varepsilon_x\)} 
   \If{not defined \(h^*\)}
  \State \(h^* \gets (a + b)/2\)
  \State Compute \(\mu_{h^*}\) for the final midpoint
\EndIf
  \State \Return \((h^*,\,\mu_{h^*})\)
\end{algorithmic}
\end{algorithm}

  \begin{proposition} \label{prop:compute}
   Suppose the assumptions of Proposition \ref{Prop:existence2} hold. Let $\{h_n\}$ be the sequence generated by Algorithm 1 with $h'$ and $h''$ as defined in Proposition \ref{Prop:existence2} and $\varepsilon_{x} = \varepsilon_{y} = 0$. Then, $\{h_n\}$ converges to $h^*$, and $\mu_{h^*}$ is an invariant distribution of $Q$.
   \end{proposition}

We note that under the conditions of Theorem \ref{Theorem: unique}, Algorithm 1 finds the unique invariant distribution of $Q$. 
In this case, it is typically immediate to compute the points $h'$ and $h''$ by using the monotonicity conditions. For example, consider the finite case $S = \{s_{1} ,\ldots , s_{n} \}$ with the standard order $s_{i} \geq s_{j}$ whenever $i \geq j$ and $\mathcal{P}(S)$ endowed with the standard stochastic dominance order $\succeq_{SD}$. Then $h'$ and $h''$ can be easily computed by applying the function $H$ to the Dirac measure centered on $s_{n}$  and the Dirac measure centered on $s_{1}$. For example, if $H (\mu)$ is the expected value operator, i.e., $H(\mu) = \sum _{s \in S} s \mu(\{s \}) $, then $h'=s_{1}$ and $h''=s_{n}$. Hence, the initial interval for the algorithm is $[s_{1},s_{n}]$. 

As an illustration, consider Example \ref{Example:convergence} where we provided a simple Markov chain that does not converge to the unique invariant distribution. We first identify the interval $[0,1]$ and $h' = h_{1}=0$, $h'' = h_{2}=1$ as explained above.  It is immediate that $H(\mu_{h}) = 1-h$, and hence, $f(h) = h - (1-h) = 2h-1$. Thus, the algorithm generates $h_{3} = 1/2$ which is the root of $f$ so the algorithm converges in the first iteration and the unique invariant distribution is $\mu_{h_{3}}$.

For a finite state with $n$ variables, the method described in Algorithm 1 is computationally efficient and straightforward to implement. In each iteration, the algorithm solves a linear equation with $n$ variables and $n+1$ constraints (enforcing that $\mu_{h}$ is a probability measure) to find the invariant distribution $\mu_{h}$. Then, the function $f$ is evaluated to proceed with the bisection method.\footnote{In particular, in each bisection step the aggregator is fixed and we solve a linear system 
rather than the original nonlinear fixed‑point problem. In terms of complexity, for a finite state space of size $n$, using standard Gaussian elimination the computational cost is $O(n^3)$ per bisection step. The bisection itself needs at most 
$
\bigl\lceil\log_2\!\bigl((h''-h')/\varepsilon_{\text{tol}}\bigr)\bigr\rceil
$
iterations to reach accuracy~$\varepsilon_{\text{tol}}$, so the overall complexity of Algorithm \ref{alg:bisection_simplified} is  
$O\!\bigl(n^3\log((h'' - h') /\varepsilon_{\text{tol}})\bigr)$. Hence, Algorithm~\ref{alg:bisection_simplified} is efficient: it is polynomial in the number of states and logarithmic in the desired precision.
} In Section~\ref{sec:Non-contraction}, we illustrate that Algorithm~\ref{alg:bisection_simplified} efficiently computes the unique invariant distribution in both our strategic queueing application and our wealth distribution application.


This approach is consistent with many well-known algorithms for solving hard optimization problems, where each iteration involves solving a simpler subproblem. For example, in cutting-plane methods to solve integer programming problems, each iteration requires solving a linear program to refine the feasible region. In the case of Algorithm \ref{alg:bisection_simplified}  described above, each iteration requires solving a linear equation in order to find the solution  of the nonlinear equation that describes the invariant distribution of the nonlinear Markov kernel $Q$.

 \subsection{Local Results} \label{sec:localresults}

In this section, we present a localized version of Theorem \ref{Theorem: unique}. Rather than applying the monotonicity conditions and Properties (U) and (C) to all probability measures as in Theorem 1, we introduce localized versions of these conditions that apply only in certain regions of the probability space. These local versions pertain only to a particular subset of probability measures that have specific interest. These conditions ensure that, within this subset, $Q$ has at most one invariant distribution. This subset may encompass probability measures that naturally emerge as candidates for invariant distributions or probability measures that are relevant for an application of interest. For a non-empty subset $\mathcal{W}$ of $\mathcal{P}(S)$ let $\mathcal{H}_{\mathcal{W}} = \{ H(\mu) : \mu \in \mathcal{W} \}$ .

\begin{definition} Let $\mathcal{W}$ be a non-empty subset of $\mathcal{P} (S)$  

(i) We say that $Q$ satisfies Property (U) on $\mathcal{W}$  if for any $h \in \mathcal{H}_{\mathcal{W}}$, the operator $M_{h}$ has a unique fixed point $\mu _{h}$. 

(ii) We say that $Q$ satisfies Property (C) on $\mathcal{W}$
if $Q$ satisfies Property (U) on $\mathcal{W}$ and  $M_{h}^{n} \theta $ converges weakly to $\mu _{h}$ for any probability measure $\theta  \in \mathcal{W}$ and any $h \in \mathcal{H}_{\mathcal{W}}$. 
\end{definition}

Similarly, we provide local versions for the monotonicity and preservation properties introduced in Section \ref{sec:definition}.

\begin{definition}
Let $D \subseteq \mathbb{R}^{S}$. 

We say that $Q$ is $D$-decreasing on $\mathcal{W}$ if for each $x \in S$, we have $Q (x ,h_{1} , \cdot ) \succeq  _{D}Q (x , h_{2}, \cdot )$ whenever $h_{2} \geq h_{1}$, $h_{1},h_{2} \in \mathcal{H}_{\mathcal{W}}$.

We say that $Q$ is $D$-preserving on $\mathcal{W}$ if for all  $h \in \mathcal{H}_{\mathcal{W}}$ the function 
\begin{equation*}
    v(x):=\int f (y) Q(x ,h ,dy) 
\end{equation*}
is in $D$ whenever $f \in D$.

\end{definition}

The following Proposition generalizes Theorem \ref{Theorem: unique}. 

\begin{proposition} \label{Prop:local}

    Let $\mathcal{W}$ be a non-empty subset of  $\mathcal{P}(S)$. 
    Let $D \subseteq \mathbb{R}^{S}$ be a non-empty set such that $H$ is increasing with respect to $\succeq  _{D}$ on $\mathcal{W}$.

    Assume that $Q$ is $D$-preserving on $\mathcal{W}$ and $D$-decreasing on $\mathcal{W}$.

Suppose that $M_{h} \theta \in \mathcal{W}$ whenever $\theta \in \mathcal{W}$ and $h \in \mathcal{H}_{\mathcal{W}}$. 

Assume that either of the following conditions hold: 

(i) $Q$ satisfies Property (C) on $\mathcal{W}$ and $\succeq_{D}$ is closed with respect to weak convergence. 

(ii) $Q$ satisfies Property (U)  on $\mathcal{W}$ and $(\mathcal{W},\succeq_{D})$ is a complete lattice.

Then $Q$ has at most one invariant distribution on $\mathcal{W}$. 
\end{proposition}

The proof of Proposition \ref{Prop:local} is similar to the proof of Theorem \ref{Theorem: unique} and is given in the Appendix.

\section{Applications} \label{sec:applications}

In this section we present our applications. In Section \ref{sec:queueing} we study the invariant distribution of a G/G/1 queueing system where arrivals depend on the expected waiting times. In Section \ref{sec:inventory} we study a dynamic inventory competition model. In Section \ref{Sec:nonlinear} we study nonlinear equations that do not necessarily satisfy contraction properties and have a unique solution. In Section \ref{sec:wealth} we study the invariant distribution of wealth distributions in dynamic economies where the rate of returns depends on the aggregate savings in the economy. 

\subsection{Strategic Behavior in Queueing Systems} \label{sec:queueing}
 A considerable body of literature exists on strategic behavior in queueing systems. Within this domain, the inter-arrival times often depend on the queue length or expected waiting time, as agents, being strategic, can opt not to join the queue if they foresee an extended waiting period  \citep{hassin2003queue}.  Typically, queueing systems are examined in the steady state, making it essential to study the existence of a unique steady state generated by the system to obtain robust comparative statics results that do not depend on the specific choice of equilibrium. We will now demonstrate how Theorem \ref{Theorem: unique} can be used to establish that there is a unique invariant distribution for the waiting time distribution within a general $G/G/1$ strategic queueing system, wherein the inter-arrival times are contingent on the expected waiting time.\footnote{Other nonlinear Markov chains were analyzed in the strategic queueing literature. For example, \cite{xu2013supermarket} show that a supermarket game where customers strategically choose which queue to join has a unique equilibrium under certain monotonicity conditions. See  \cite{mukhopadhyay2016randomized} and \cite{yang2018mean} for further related models. }

Consider a $G/G/1$ queue where the time between the $t$th and $t+1$th arrivals is given by the random variable $T_{t}$ and the service time of the $t$-th customer is given by the random variable $S_{t}$. Because agents are strategic they are less likely to join the queue when the waiting time is longer. We assume that the time between arrivals depends on the expected waiting time,\footnote{Announcing average waiting times to customers is a common practice in queue management \citep{bassamboo2021general}, particularly in environments like theme parks, where it helps manage crowd flow and set visitor expectations. In practice, posted wait times can be calculated using a variety of factors and are not solely based on the distribution of the last agent's waiting time.} represented as 
$ T_t(  \mathbb{E}(X_t) ) $, where $ X_t $ is the waiting time of the $ t $th customer. 
To capture that when the expected waiting time increases, fewer agents join the queue, we assume that $ T_t(h) \succeq_{SD} T_t(h') $ whenever $ h \geq h' $, for $ h, h' \in \mathbb{R}_{+} $. 
In other words, the time between arrivals becomes stochastically longer as the expected waiting time rises.  We assume that $
(S_{t})_{t \in \mathbb{N}}$ are identically distributed and independent random variables with positive finite expectations and finite variances, and $T_{t}(h)$ has bounded first two moments, is continuous, and $\{ T_{t}(h) \}$ are independent random variables across time for each $h \geq 0$. We also assume $\mathbb{E}T_{t}(0) > \mathbb{E}S_{t}$ so the (linear) G/G/1 queueing system for a fixed $h$ is stable. 

The expected waiting times experienced by customers in the queue evolve by the following nonlinear Markov chain on $\mathbb{R}_{+}$: 
\begin{equation} \label{Eq:queue} X_{t+1} = \max (0, X_{t} + S_{t} - T_{t} (\mathbb{E}(X_t) ) ) .
\end{equation}

It can be easily verified that $Q$ is $D$-preserving and $D$-decreasing when $D$ is the set of all increasing functions. Under the 
 assumption stated above that the queue does not explode, i.e., $\mathbb{E}S_{t} < \mathbb{E}T_{t}(0)$, a standard argument from the Markov chain literature (e.g., Theorem 
19.3.5 in \cite{meyn2012markov}) shows that Property (C) holds. Hence, we can use Theorem \ref{Theorem: unique} to conclude that there exists at most one waiting time equilibrium steady state distribution. We show that existence of an invariant distribution follows from Proposition \ref{Prop:existence2}. The proofs of all the Corollaries are deferred to Section \ref{Sec:CorrProofs} in the Appendix. 

\begin{corollary} \label{Corr:Queue}
    The nonlinear Markov chain describing the queueing system in Equation (\ref{Eq:queue}) has a unique invariant distribution. 
\end{corollary}

As a particular example, we study  an M/G/1 queueing system where the arrival rate depends on the expected waiting time and provide a closed-form expression for the stationary expected waiting time.

\begin{example} ($M/G/1$ queue). 
Consider an $M/G/1$ queue so the time between arrivals has an exponential distribution. Let  $Law(S_{t}) = Law (S)$ and $Law(T_{t}(h)) $ has an exponential distribution with the parameter $\lambda(h)$. Suppose that the mean interarrival time equals the expected waiting time so $\lambda(h) = 1/h$. 

\begin{claim} \label{claim_queue} There is a unique invariant distribution for the nonlinear Markov chain given in Equation (\ref{Eq:queue}) and the expected value of the stationary waiting time $X_{\infty}$ is given by the closed-form expression
$$ \mathbb{E}(X_{\infty}) = \frac{\mathbb{E}(S^{2})} {\sqrt{\mathbb{E}(S)^{2}+2\mathbb{E}(S^{2})} - \mathbb{E}(S)} .$$

In particular, if the queue is an $M/M/1$ queue so $S$ is an exponential random variable then 
$$ \mathbb{E}(X_{\infty}) = \frac{ 2\mathbb{E}(S)} {\sqrt{5}-1}. $$
    
\end{claim}


\end{example}

In Section~\ref{sec:Non-contraction}, we provide a numerical illustration showing that the nonlinear Markov chain describing the strategic queueing system in Equation~(\ref{Eq:queue}) is not a contraction under standard metrics, even in the simple M/M/1 case. We then deploy Algorithm~\ref{alg:bisection_simplified} to compute the unique invariant waiting time distribution.

\subsection{Dynamic Pricing and Inventory Replenishment} \label{sec:inventory}
There is a rich body of work in operations on dynamic retail inventory competition under stochastic demand (e.g., \cite{liu2007dynamic}, \cite{adida2010dynamic} \cite{olsen2014markov}, \cite{bansal2022monge} to name a few). We contribute to this literature by establishing conditions that guarantee the existence of a unique invariant distribution in a general inventory dynamics model and by introducing an algorithm to compute this distribution, even in the presence of nonlinearities and the absence of contraction. 

Consider a population of ex-ante identical retailers indexed by $j$. Each retailer faces stochastic demand and chooses both a price and a replenishment quantity each
period. 
The state of each retailer at time \( t \) is given by its inventory level \( X_{t}^{j} \in \{0,1,\ldots,C\} \), where \( C > 0\) is a fixed capacity. Let \(\mu_{t}\) be the law of \( X_{t}^{j} \) across all retailers. We define a continuous aggregator \( H(\mu_{t}) \) that maps the distribution of inventory levels to a real number. For instance, \( H(\mu_{t}) \) could be the average inventory across all retailers and is assumed to be increasing with respect to stochastic dominance. In applications, the aggregator can also be derived from the underlying model, such as in stockout-based inventory substitution \citep{olsen2014markov} where the demand for an out-of-stock product from one retailer is reallocated to substitute products offered by another retailer based on predefined substitution probabilities, creating a dependency between the inventories of different retailers.

At the beginning of period \( t \), each retailer observes its own inventory \( X_{t}^{j} \) and the aggregator \( H(\mu_{t}) \), and sets a price \( p_{t}^{j} = \pi(X_{t}^{j}, H(\mu_{t})) \) given 
 some continuous pricing policy $\pi$. 
Given this price and the aggregator, the demand \( D_{t}^{j} \) faced by retailer \( j \) in period \( t \) is a discrete non-negative random variable with a distribution that depends on both \( p_{t}^{j} \) and \( H(\mu_{t}) \). After sales occur, the retailer replenishes a quantity \( g(X_{t}^{j}, H(\mu_{t})) \) units of inventory, where \( 0 \leq g(X_{t}^{j}, H(\mu_{t})) \leq C - X_{t}^{j} \) ensures that the next period's inventory does not exceed capacity and the replenishment policy $g$ is  assumed to be continuous.\footnote{Pricing and replenishment policies are well studied for such a setting under different demand models, e.g.,  \citep{chen2019coordinating, chen2021nonparametric, keskin2022data}.} For simplicity, we assume that all retailers share the same pricing policy function, which depends on their current inventory levels and the aggregator, as well as the same demand structure, which is determined by their price and the aggregator. However, it is straightforward to extend the model by introducing retailer-specific types that influence both their policy functions and demand structures, allowing for ex-ante heterogeneity across retailers. Thus, $D_{t}^{j} ( p , h)$ is independent and identically distributed across
time and across agents given the price and aggregator.

Hence, the inventory evolves according to the nonlinear Markov chain:\footnote{We note that the state recursion in \eqref{eq:inventory} mirrors the inventory dynamics in earlier works such as \citet{liu2007dynamic}, \citet{adida2010dynamic}, and \citet{olsen2014markov}, in the sense that next-period inventory equals current inventory minus stochastic demand (which depends on competitors) plus a replenishment order. These papers typically consider duopolies or small oligopolies and study Nash or Markov-perfect equilibria in that setting. In contrast, we analyze a distribution-dependent version of the problem, where retailers interact only through an aggregator and focus on conditions that imply the uniqueness of a stationary equilibrium.
} 
\begin{equation}\label{eq:inventory}
X_{t+1}^{j} = (X_{t}^{j} - D_{t}^{j}(\pi(X_{t}^{j},H(\mu_{t}) ), H(\mu_{t}) )) _{+} + g(X_{t}^{j}, H(\mu_{t}))
\end{equation}
where $(x)_{+} = \max(x,0)$. A stationary equilibrium for this model  corresponds to an invariant distribution of the nonlinear Markov chain described by the equation above. The
equilibrium represents a stable long-run configuration of inventories and aggregator values. 

We can apply Algorithm \ref{alg:bisection_simplified} to find the equilibrium of the system. In addition, under suitable monotonicity conditions we now present, Theorem  \ref{Theorem: unique} can be applied to ensure the uniqueness of this stationary equilibrium. The proof of the following Corollary follows immediately from Proposition \ref{Prop:existence} and Theorem \ref{Theorem: unique} so we omit it.  We will  write $D(\pi(x,h),h)$ to describe the dependence of the random demand on the pricing policy and aggregator. 

\begin{corollary} \label{corr:inventory}
    Suppose that $H$ is increasing with respect to stochastic dominance and the following two conditions hold: 

    (1) The function $$f(x,h):=  \Pr \left [  (x - D(\pi(x,h ), h)) _{+} + g(x, h) \geq c \right]$$ is increasing in $x$ and decreasing in $h$ for each $c = 0 ,\ldots ,C$. 

    (2) The linear Markov chain 
    $$ X_{t+1} = (X_{t} - D_{t}(\pi(X_{t},h ), h )) _{+} + g(X_{t}, h)$$
    has a unique stationary distribution for each $h$. 

Then the nonlinear Markov chain describing the inventory system in Equation (\ref{eq:inventory}) has a unique invariant distribution. 
\end{corollary}
The first condition in Corollary \ref{corr:inventory} guarantees that $Q$ is $D$-preserving and $D$-decreasing. Specifically, $f$ increasing in $x$ means, intuitively, that a higher current inventory 
 makes higher future inventory levels more likely. 
$f$ decreasing in $h$ means, intuitively, that market saturation, in terms of inventories, reduces the probability of higher future inventory levels. Overall, while these conditions may not hold in some models, they are intuitive for practical settings.

The second condition is technical in nature and guarantees that property (U) holds. It is  easy to establish using standard irreducibility arguments for finite Markov chains when  there is sufficient randomness or variability in the demand and replenishment policies.

\subsection{Nonlinear Equations} \label{Sec:nonlinear}

The study of nonlinear systems of equations in $\mathbb{R}^{n}$ has long been a significant area of interest in mathematics and its applications. Finding conditions that ensure a unique solution to such systems is crucial as it offers insights into the properties and stability of solutions, which in turn, have far-reaching implications across various fields, including operations, engineering, economics, and optimization.  
It is generally uncommon to identify a comprehensive set of conditions that guarantee a unique global solution for a system of nonlinear equations in $\mathbb{R}^{n}$ that do not satisfy contraction properties. We apply Theorem \ref{Theorem: unique} to determine conditions that ensure a unique solution for a specific class of nonlinear equations, which we define subsequently. These conditions are based  on monotonicity concerning the majorization order as opposed to typical approaches that require contraction.

Let $\Delta_{n} = \{ \boldsymbol{x} \in \mathbb{R}^{n}: \sum_{i=1}^{n}  x_{i} = 1, x_{i} \geq 0 \text{ }\forall i \} $ be the $n$-dimensional simplex.  
Consider a stochastic matrix $\boldsymbol{P} (G(\boldsymbol{x})) \in \mathbb{R}^{n\times n}$ that is parameterized by $G(\boldsymbol{x})$ where  $G:\Delta _{n} \rightarrow A$  and $A \subseteq \mathbb{R}$ is the image of $G$, i.e., $P_{ij} (a) \geq 0$, and $\sum _{j=1}^{n} P_{ij} (a) = 1$ for all $a \in A$. We assume that $G$ is a continuous function.

For $\boldsymbol{x},\boldsymbol{y} \in \mathbb{R}^{n}$ write $\boldsymbol{x} \geq _{m} \boldsymbol{y}$ if $ \sum _{j=k}^{n} x_{j} \geq \sum _{j=k}^{n} y_{j}$ for all $1 \leq k \leq n$ and $\sum_{j=1}^{n} x_{j} = \sum _{j=1}^{n} y_{j}$ (the order $\geq _{m}$ is sometimes called majorization between vectors in $\mathbb{R}^{n}$). We denote by $\boldsymbol{P}_{i}(a)$ the $i$th row of the matrix $\boldsymbol{P}$.

The following Corollary follows  from applying Theorem \ref{Theorem: unique} and Proposition \ref{Prop:existence}. 

\begin{corollary} \label{Coro:nonlinear}
Let $G:\Delta _{n} \rightarrow A$ be a continuous function that  is increasing with respect to $ \geq _{m}$. The nonlinear system of equations $\boldsymbol{x} = \boldsymbol{x} \boldsymbol{P} (G(\boldsymbol{x}))$ on $\Delta_{n}$ where $\boldsymbol{P} (G(\boldsymbol{x}))$ is a stochastic matrix that is parameterized by $G(\boldsymbol{x})$ has a unique solution if the following three conditions hold:

(1) For all $a \in A$, $i \geq i'$, we have  $\boldsymbol{P}_{i}(a) \geq _{m} \boldsymbol{P}_{i'}(a)$.

(2) For all $1 \leq  i \leq n$, $a' \geq a$, $a,a' \in A$, we have  $\boldsymbol{P}_{i}(a) \geq _{m} \boldsymbol{P}_{i}(a')$.


(3) For all $a \in A$, the linear system of equations $\boldsymbol{x} = \boldsymbol{x} \boldsymbol{P} (a)$ for $\boldsymbol{x} \in \Delta_{n}$ has a unique solution.

\end{corollary}

Corollary \ref{Coro:nonlinear} yields a simple recipe for producing nonlinear equations with unique solutions, even in settings where no global contraction can be exhibited.  For instance, we can take 
$
G(x)\;=\;\sum_{i=1}^n c_i\,x_i$ with 
$c_1\le c_2\le\cdots\le c_n,
$
so that $G(x)$ is increasing with respect to $\geq_{m}$ and the corresponding quadratic system
$\boldsymbol{x} = \boldsymbol{x} \boldsymbol{P} (G(\boldsymbol{x}))$ on $\Delta_{n}$  has a unique solution in the simplex if it 
satisfies the monotonicity conditions of Corollary \ref{Coro:nonlinear} that are typically easy to check. This is despite the fact that the map $\boldsymbol{x} \boldsymbol{P} (G(\boldsymbol{x})) $ need not be a contraction. 


\subsection{Wealth Distributions} \label{sec:wealth}

In heterogeneous-agents macroeconomic models (see \cite{stachurski2009economic} for a recent textbook treatment of economic dynamic models), agents determine their consumption, savings, and allocation of savings across financial assets based on their current wealth level in each period.

An extensive literature exists on these models, specifically focusing on the analysis of stationary equilibria and the associated stationary wealth distributions. Despite the vast body of research, the conditions ensuring the uniqueness of equilibrium are restricted to a handful of special cases.\footnote{For instance, see \cite{light2020uniqueness, light2023general, achdou2022income}. } In this section, we employ Theorem \ref{Theorem: unique} to prove the uniqueness of a stationary equilibrium under a typical progression of wealth dynamics in these models, given that agents' savings increase with the rate of returns and their current wealth levels. We proceed to outline the model.

In each period $t$, there are $n$ non-negative random variables $R_{1,t},\ldots,R_{n,t}$ with bounded supports $[0, \overline{r}]$ that represent returns from different financial assets $i=1,\ldots,n$.  
The random return $R_{i,t}$ of asset $i$ is parameterized by a continuous aggregator $H(\mu)$ and we write $R_{i,t}(H(\mu))$ to capture this dependence. The aggregator is a function of the wealth distribution in the economy $\mu$ and  is increasing with respect to stochastic dominance. In many applications the aggregator is given by the total savings or wealth in the economy (e.g., \cite{aiyagari1994}). We assume that $R_{i,t}(h)$ is independent and identically distributed across time for each $i=1,\ldots,n$ and each $h$. For notational simplicity we sometimes write $R_{i}(h)$ instead of $R_{i,t}(h)$ to describe the random return of asset $i$ given the aggregator, and we assume that $R_{i}(h)$ is continuous for $i=1,\ldots,n$.

Each agent has a Markovian policy $\boldsymbol{g} = (g_{1},\ldots,g_{n})$, which is a vector of functions that determines how wealth is allocated across assets. Specifically,  $g_{i}(R_{1}(\mu),\ldots,R_{n}(\mu),x)$ represents the non-negative amount that an agent with wealth $x$ allocates to asset $i$ when the current  returns are given by $(R_{1}(\mu),\ldots,R_{n}(\mu))$. More formally, 
let \(\mathcal{T}\) denote the space of random variables with support on \([0, \overline{r}]\) then each function \( g_i : \mathcal{T}^n \times \mathbb{R}_+ \to \mathbb{R}_+ \) determines the allocation to asset \( i \) based on the returns and the agent’s wealth.\footnote{We assume for simplicity that the agents policy function depends on their current  wealth and returns only. All the results in this section can be easily extended to the case when each agent uses a different policy that depends on the agent's specific features such as preferences or behavioral biases.} In applications, the agent's policy is typically derived from a consumption-saving dynamic programming problem.  In our analysis, we assume a general policy function that can be deduced from rational agents, behavioral biases  \citep{acemoglu2024equilibrium}, myopic agents, or learning algorithms. We assume that $g_{i}$ is continuous for $i=1,\ldots,n$.

 In each period $t$, each agent $j$ receives a non-negative random income $Y_{t}^{j}$ that is independent and identically distributed across time and across agents and has a bounded support $[0,\overline{y}]$. Note that the returns $R_{i,t}(h)$ depend on the wealth distribution in the economy and are common to all agents while the random income $Y_{t}^{j}$ represents agent-specific noise.

Each agent's wealth evolution is described by the following nonlinear Markov chain:
\begin{equation} \label{eq:wealth}
X^{j}_{t+1} = \sum _{i=1}^{n} g_{i}(R_{1}(H(\mu _{t} )),\ldots , R_{n} (H(\mu_{t}) ) ,X^{j}_{t} )R_{i}(H(\mu _{t}) ) + Y ^{j}_{t+1} 
\end{equation}
 where $X_{t}^{j}$ is the current wealth agent $j$ has,  and $\mu_{t}$ is the law of $X_{t}^{j}$ which describes the  wealth distribution across agents in period $t$. Thus, if an agent has a current wealth of  $x_{t}$, the agent allocates $g_{i}$ to asset $i$, then the next period's wealth is given by the sum of the returns on these investments plus the income received in the next period. A stationary equilibrium in this economy is defined by an invariant distribution of the nonlinear Markov chain given in Equation (\ref{eq:wealth}) with the interpretation that this distribution represents the long run equilibrium wealth distribution across agents \citep{aiyagari1994,acemoglu2012}. 

 Under standard assumptions, the policy function is increasing in the current wealth, i.e., savings increase when the agent's wealth is higher, and the returns are decreasing in the savings with respect to first order stochastic dominance, i.e., the returns are (stochastically) lower when the total savings are higher (see \cite{acemoglu2012}, and \cite{acemoglu2024equilibrium}). Under these assumptions, we can apply Theorem \ref{Theorem: unique} to conclude that there is at most one stationary wealth distribution equilibrium if the total amount of savings $\sum g_{i}$ is increasing in the rate of returns. In the economics literature, this property means that the substitution effect dominates the income effect. Hence, the key condition that implies that there is at most one stationary wealth distribution  equilibrium is that  savings increase with the rate of returns. We now present this result formally. 

 \begin{corollary} \label{Coro:wealth}
     Suppose that $H(\mu)$ is increasing with respect to $\succeq_{SD}$ and assume that $g_{i} \leq M$ for some $M$ for each $i$.\footnote{The assumption that $g_{i}$ is bounded is used only to prove existence. We note that the existence of the stationary wealth distribution equilibrium is widely studied in the literature (e.g., \cite{acikgoz2015existence},  \cite{acemoglu2012}, \cite{zhu2020existence}, and \cite{light2022mean}) where the boundedness of $g_{i}$ can be established by considering the consumption-savings dynamic programming problem the agents' solve or by assuming an exogenous savings bound.} Assume that:

     (1) Property (C) holds.\footnote{There is a vast literature on conditions that ensure that Property (C) holds in different models of wealth dynamics by employing results from the standard Markov chain literature. For recent results see \cite{ma2020income}.}
     
     (2) The function $\sum g_{i}$ is increasing in $x$ and decreasing in the aggregator in the sense that 
     $$ \sum _{i=1}^{n} g_{i} (R_{1} (h_{2} ), \ldots , R_{n}  ( h_{2} ) , x_{2} ) \geq  \sum _{i=1}^{n} g_{i} (R_{1} (h_{1} ), \ldots , R_{n}  ( h_{1} ) , x_{1} )  $$
     whenever $x_{2} \geq x_{1}$ and $h_{1} \geq h_{2}$. 

     (3) For $i=1,\ldots,n$, $R_{i}(h_{2}) \succeq_{SD} R_{i}(h_{1})$ whenever $h_{1} \geq h_{2}$. 

     Then the nonlinear Markov chain  described in Equation (\ref{eq:wealth}) has a unique invariant distribution. 
 \end{corollary}

A special case of the last result with one financial asset  that has a constant interest rate and rational agents is the model by \citep{aiyagari1994}. Uniqueness for this model is studied in \cite{light2020uniqueness} which establishes all the conditions presented in Corollary \ref{Coro:wealth} for the case where agents maximize expected utility with constant relative risk aversion coefficient that is less than or equal to $1$. In Section~\ref{sec:noncontraction:wealth}, we numerically solve such a model, demonstrate that the resulting nonlinear Markov chain is generally not contractive, and compute its invariant wealth distribution using Algorithm~\ref{alg:bisection_simplified} (Section \ref{sec:compute}).

\section{Conclusions} 

This paper studies discrete-time nonlinear Markov chains with an aggregator and establishes  conditions that imply the uniqueness and existence of an invariant distribution for these chains. Unlike traditional approaches that rely on contraction properties of the chains, our conditions leverage monotonicity properties and the aggregator structure to establish uniqueness.
 We provide a computational method to compute the invariant distribution and apply our results to different settings including strategic queueing systems, inventory dynamics,  nonlinear equations, and the evolution of wealth distributions in dynamic economies. We believe that our results can be applied to other models where the flexible monotonicity conditions we provide are naturally satisfied.

Important open questions remain concerning nonlinear Markov chains. For instance, our examples show that even in a simple two-state chain, convergence to the invariant distribution is not guaranteed even when it is unique. Therefore, developing algorithms that ensure convergence to an invariant distribution in nonlinear Markov chains without an aggregator, i.e., settings beyond the scope of the bisection method introduced in Section~\ref{sec:compute}, remains an important direction for enabling practical computation in these models.


\section{Appendix}

\subsection{Proofs of Theorem \ref{Theorem: unique} and Propositions \ref{Prop:existence}, \ref{Prop:existence2}, \ref{prop:compute},   \ref{Prop:local}}
We will use the following Proposition to prove Theorem \ref{Theorem: unique} (see Corollary 2.5.2 in \cite{topkis2011supermodularity}). 

\begin{proposition}
\label{Topkis Fixed point}Suppose that $Z$ is a non-empty complete lattice, $E$ is a partially ordered set, and $f$ is an increasing function from $Z \times E$ into $Z$. Then, for each $e \in E$, the greatest and least fixed points 
of $f$ exist and are increasing in $e$ on $E$. 
\end{proposition}

\begin{proof}[Proof of Theorem \ref{Theorem: unique}]
Let $\theta _{1} ,\theta _{2} \in \mathcal{P} (S)$ and assume that $\theta _{1}  \succeq  _{D}\theta _{2}$. 
Let $\mu_{1} ,\mu_{2}$ be two invariant distributions of $Q$. Assume without loss of generality that $h_{2} := H( \mu_{2}) \geq H( \mu_{1}) :=h_{1}$ and let $f:S \rightarrow \mathbb{R}$ be a function such that $f \in D$. We have 
\begin{align*}\int _{S}  f(x) M_{h_{2}} \theta _{2} (dx)  &  =\int _{S} \int_{S} f (y) Q(x ,h_{2} ,dy)   \theta _{2} (d x) \\
 &  \leq  \int _{S} \int_{S} f (y) Q(x ,h_{1},dy)   \theta _{2} (d x) \\
 &  \leq \int _{S} \int_{S} f (y) Q(x ,h_{1} ,dy)   \theta _{1} (d x) \\
 & = \int _{S}  f(x) M_{h_{1}} \theta _{1} (dx). 
 \end{align*}

 Thus, $M_{h_{1}} \theta _{1}  \succeq  _{D}M_{h_{2}} \theta _{2}$. The first inequality follows from the fact that $Q$ is $D$-decreasing. The second inequality follows from the facts that $\theta _{1}  \succeq  _{D}\theta _{2}$ and $Q$ is $D$-preserving. 
We conclude that $M_{h_{1}}^{n} \theta _{1}  \succeq  _{D}M_{h_{2}}^{n} \theta _{2}$ for all $n \in \mathbb{N}$. 

Assume that condition (i) of the theorem holds. 
The fact that $Q$ satisfies Property (C) implies that $M_{h_{i}}^{n} \theta _{i}$ converges weakly to the unique fixed point of $M_{h_{i}}$ which is given by $\mu _{h_{i}}$ for $i=1,2$. Because $\mu_{1}$ and $\mu_{2}$ are invariant distributions of $Q$ we have  $\mu _{h_{i}} =\mu_{i}$ for $i=1,2$. Because $ \succeq  _{D}$ is closed with respect to weak convergence,  we have $\mu_{1} \succeq  _{D} \mu_{2}$. Using the fact that $H$ is increasing with respect to $\succeq  _{D} $ implies $h_{1} \geq h_{2}$. 

We conclude that if $\mu_{1}$ and $\mu_{2}$ are invariant distributions of $Q$ then $H(\mu_{1}) = H(\mu_{2})$. Thus, $Q (x ,H(\mu_{1}) ,B) =Q (x,  H(\mu_{2}) ,B)$ for all $x \in S$ and $B \in \mathcal{B} (S)$. Because $Q$ satisfies assumption (U) the operators $M_{H(\mu_{1})}$ and $M_{H(\mu_{2})}$ have unique fixed points. Thus, $\mu _{H(\mu_{1})} = \mu _{H(\mu_{2})}$, i.e., $\mu_{1} = \mu_{2}$. 
We conclude that if an invariant distribution of $Q$ exists, it is unique. 

Now assume that condition (ii) of the theorem holds. Define the order \(\geq'\) as the reverse of the usual order $\geq$: for \(x, y\), we write \(x \geq' y\) if and only if \(y \geq x\).  Under this assumption, the arguments above imply that 
the operator $M$ is increasing from $\mathcal{P}(S) \times \mathcal{H}$ to $\mathcal{P}(S)$ on the complete lattice $( \mathcal{P}(S),\succeq  _{D})$ when $\mathcal{H}$ is endowed with $\geq '$. 
Then by applying Proposition \ref{Topkis Fixed point} to the increasing operator $M$  we have $\mu _{h_{1}} \succeq _{D} \mu _{h_{2}}$, i.e., $\mu_{1} \succeq  _{D} \mu_{2}$. Now we can use the same arguments as the arguments for the case that condition (i) holds to show that if an invariant distribution of $Q$ exists, it is unique. 
\end{proof}

In order to establish the existence of an invariant distribution we will use Schauder-Tychonoff's following  fixed-point theorem (see Corollary 17.56 in \cite{aliprantis2006infinite}).

\begin{proposition}
(Schauder-Tychonoff) Let $K$ be a non-empty, compact, convex subset of a locally convex Hausdorff space, and let $f :K \rightarrow K$ be a continuous function. Then the set of fixed points of $f$ is compact and non-empty. 
\end{proposition}

\begin{proof}[Proof of Proposition \ref{Prop:existence}]
  Because $S$ is a compact Polish space  $\mathcal{P} (S)$ is a compact Polish space under the weak topology  (see Theorem 15.11 in \cite{aliprantis2006infinite}).
Clearly $\mathcal{P} (S)$ is convex. $\mathcal{P} (S)$ endowed with the weak topology is a locally convex Hausdorff space. Thus, if $T $ is continuous, we can apply Schauder-Tychonoff's fixed point theorem to conclude that $T $ has a fixed point. 

To show that $T$ is continuous, take a sequence of measures $\{ \mu_{n} \}$ and assume that it converges weakly to $\mu$.  

Let $f :S \rightarrow \mathbb{R}$ be a continuous and bounded function. Because $Q$ and $H$ are continuous we have $\lim _{n \rightarrow \infty} \int_{S} f(y) Q(x_{n},H(\mu_{n}),dy) = \int_{S} f(y) Q(x,H(\mu),dy) $ whenever $x_{n} \rightarrow x$. Define $m_{n}(x) : = \int_{S} f(y) Q(x,H(\mu_{n}),dy) $. Then $m_{n}(x)$ is a uniformly bounded sequence of functions such that $m_{n}(x_{n}) \rightarrow m(x)$ whenever $x_{n} \rightarrow x$. Thus, by Lebesgue's Convergence Theorem for varying measures (see Theorem 3.5 in \cite{serfozo1982convergence} and Section 5 in \cite{feinberg2020fatou}) we have $\lim _{n \rightarrow \infty} \int m_{n}(x) \mu_{n}(dx) = \int m(x) \mu (dx)$. Hence,
\begin{align*}\underset{n \rightarrow \infty }{\lim }\int _{S}f (x) T  \mu_{n} (d x) &  =\underset{n \rightarrow \infty }{\lim }\int _{S} \int_{S} f(y) Q(x,H(\mu_{n}),dy) \mu_{n} (d x) \\
 &  =\int _{S} \int_{S} f(y) Q(x,H(\mu),dy) \mu (d x) \\
 &  =\int _{S}f (x) T \mu(dx) .\end{align*}
 Thus, $T\mu_{n}$ converges weakly to $T \mu$. We conclude that $T$ is continuous. Thus, by the Schauder-Tychonoff's fixed point theorem, $T $ has a fixed point. 
\end{proof}

 \begin{proof} [Proof of Proposition \ref{Prop:existence2}]
 Consider the function $f(h)= h- H(\mu_{h})$ from $[h',h'']$ to $\mathbb{R}$ which is well defined because $\mu_{h} \in \mathcal{P}(S)$ for all $h \in [h',h'']$. 

We first claim that a root of $f$, say $h^{*}$, corresponds to an invariant distribution $\mu_{h^{*}}$ of $Q$. To see this, let $h^{*}$ be a root of $f$, that is, $H(\mu_{h^{*}} ) = h^{*}$.

From Property (U), $\mu_{h^{*}}$ is the unique probability measure that satisfies
 $$\mu_{h^{*}}(B) = \int Q(x, h^{*}, B)  \mu_{h^{*}}(dx), $$
so $H(\mu_{h^{*}} ) = h^{*}$ implies that 
 $$\mu_{h^{*}}(B) = \int Q(x, H(\mu_{h^{*}} ), B)  \mu_{h^{*}}(dx), $$
i.e., $\mu_{h^{*}}$ is an invariant distribution of $Q$.

If $h'' = H(\mu_{h''})$ or $h' = H(\mu_{h'})$ then $f$ has a root, and hence, $Q$ has an  invariant distribution. If 
 $h'' > H(\mu_{h''})$ and $h' < H(\mu_{h'})$, we have $f(h'') > 0 > f(h')$ so if $f$ is continuous we can apply the intermediate value theorem to prove that $f$ has a root, that is, $Q$ has an invariant distribution. 

We will now show that $f$ is continuous to conclude the proof. 

 Consider a sequence $\{h_{n}\}$, $h_{n} \in [h',h'']$ such that  $h_{n}$ converges to $h$ and let $\{\mu_{h_{k}}\}$ be a subsequence of $\{ \mu_{h_{n}} \}$ that converges to $\lambda$.  From Lebesgue's Convergence Theorem for varying measures (see Theorem 3.5 in \cite{serfozo1982convergence}) and using the same logic as in the proof of Proposition \ref{Prop:existence}, for every continuous and bounded function $m:S \rightarrow \mathbb{R}$, we have
 \begin{align*}\underset{k \rightarrow \infty }{\lim }\int _{S}m (x) \mu_{h_{k}} (d x) &  =\underset{k \rightarrow \infty }{\lim }\int _{S} \int_{S} m(y) Q(x,h_{k},dy) \mu_{h_{k}} (d x) \\
 &  =\int _{S} \int_{S} m(y) Q(x,h,dy) \lambda (d x) \\
 &  =\int _{S}m (x) M_{h} \lambda(dx) .\end{align*}
Because $\{ \mu_{h_{k}} \}$ converges to $\lambda$ we also have
$$ \lim _{k \rightarrow \infty} \int_{S} m(x) \mu_{h_{k}}  (dx) = \int_{S} m(x) \lambda  (dx).$$
Thus, $\lambda =M_{h}\lambda$. From assumption (U), $\mu_{h} $ is the unique fixed point of $M_{h}$, and thus,  $\lambda =\mu_{h} $. 

We conclude that any subsequence of $\{ \mu_{h_{n}} \}$ that converges weakly at all converges weakly to $\mu_{h}$. Furthermore, from assumption, the sequence $\{ \mu_{h_{n}} \}$ is a tight sequence of probability measures. Thus, $\{ \mu_{h_{n}} \}$ converges weakly to $\mu _{h}$ (see the Corollary after Theorem 25.10 in \cite{billingsley2008probability}). 

Because $H$ is continuous, we conclude that $f(h) = h - H(\mu_{h})$ is continuous on $[h',h'']$ which completes the proof. 
\end{proof}

\begin{proof} [Proof of Proposition \ref{prop:compute}]
From Proposition \ref{Prop:existence2} the function $f$ is continuous and has opposite signs at $h_{1}$ and $h_{2}$. Hence, the sequence $h_{n}$ defined in the statement of the proposition converges linearly to the root of $f$ (see for example, Theorem 2.1 in \cite{burden19852}).
   
   From Proposition \ref{Prop:existence2} if $h^{*}$ is a root of $f$, then $\mu_{h^{*}}$ is an invariant distribution of $Q$ which completes the proof. 
\end{proof}

\begin{proof}[Proof of Proposition \ref{Prop:local}]
The proof is similar to the proof of Theorem \ref{Theorem: unique}. We provide it here for completeness. Let $\theta _{1} ,\theta _{2} \in \mathcal{W}$ such that $\theta _{1}  \succeq  _{D}\theta _{2}$. 
Let $\mu_{1} ,\mu_{2} \in \mathcal{W}$ be two invariant distributions of $Q$. 

Assume without loss of generality that $h_{2} := H( \mu_{2}) \geq H( \mu_{1}) :=h_{1}$ so $h_{1},h_{2} \in \mathcal{H}_{\mathcal{W}}$ and let $f:S \rightarrow \mathbb{R}$ be a function such that $f \in D$. We have 
\begin{align*}\int _{S}  f(x) M_{h_{2}} \theta _{2} (dx)  &  =\int _{S} \int_{S} f (y) Q(x ,h_{2} ,dy)   \theta _{2} (d x) \\
 &  \leq  \int _{S} \int_{S} f (y) Q(x ,h_{1},dy)   \theta _{2} (d x) \\
 &  \leq \int _{S} \int_{S} f (y) Q(x ,h_{1} ,dy)   \theta _{1} (d x) \\
 & = \int _{S}  f(x) M_{h_{1}} \theta _{1} (dx). 
 \end{align*}

 Thus, $M_{h_{1}} \theta _{1}  \succeq  _{D}M_{h_{2}} \theta _{2}$. The first inequality follows from the fact that $Q$ is $D$-decreasing on $\mathcal{W}$. The second inequality follows from the facts that $\theta _{1}  \succeq  _{D}\theta _{2}$ and $Q$ is $D$-preserving on $\mathcal{W}$. Now because $\theta_{1},\theta_{2} \in \mathcal{W}$ and $h_{1},h_{2} \in \mathcal{H}_{\mathcal{W}}$, we have $M_{h_{1}} \theta _{1}, M_{h_{2}} \theta _{2} \in \mathcal{W}$. Applying the same argument as above again, we conclude that $M_{h_{1}}^{n} \theta _{1}  \succeq  _{D}M_{h_{2}}^{n} \theta _{2}$ for all $n \in \mathbb{N}$. 

Now the proof continues exactly as in the proof of Theorem \ref{Theorem: unique}.
\end{proof}

\subsection{Proof of Corollaries \ref{Corr:Queue}, \ref{Coro:nonlinear}, \ref{Coro:wealth}} \label{Sec:CorrProofs}

\begin{proof} [Proof of Corollary \ref{Corr:Queue}]
  Let $H(\mu) =  \int _{\mathbb{R} _{+}} x \mu(dx) $, $Law(S_{t}) = Law (S)$ and $Law(T_{t} (h) ) = Law (T(h) )$. Let $D$ be the set of increasing functions, so $\succeq_{D}$ is equivalent to the first order stochastic dominance order $\succeq_{SD}$ and $H$ is increasing with respect to $\succeq_{D}$. From  Theorem 19.3.5 in \cite{meyn2012markov}, Property (C) is satisfied because $\mathbb{E}(T(h)) \geq \mathbb{E} (T(0)) > \mathbb{E} (S)$ for all $h  \geq 0$.
  
Let $f: \mathbb{R} \rightarrow \mathbb{R}$ be increasing. Because $T$ is stochastically increasing in $h$ the function 
$$ \int f(y)Q(x, h, dy) = \mathbb{E} f \left ( \max \{x + S - T(h) , 0 \} \right ) $$ 
is increasing in $x$ and decreasing in $h$ where the expectation is taken with respect to the random variables $S$ and $T(h)$. Thus, $Q$ is $D$-preserving and $D$-decreasing.
   Hence, from Theorem \ref{Theorem: unique} we conclude that the nonlinear Markov chain given in Equation (\ref{Eq:queue}) has at most one invariant distribution.

  For existence, first note that the function $H(\mu_{h})$ is bounded from below by $0$ so $h' \leq H(\mu_{h'})$ for $h'=0$. In addition, from the proof of Theorem \ref{Theorem: unique} the function $H(\mu_{h})$ is decreasing in $h$ so $H(\mu_{h}) \leq H(\mu_{0}) < \infty$ as $\mathbb{E} (T(0)) > \mathbb{E} (S)$. Hence, we can find $h'' \geq H(\mu_{h''})$ for some $h'' > 0$. 
  
  We already established that property (C) holds, and hence, property (U) holds too. 
  Further, it is immediate to verify that $H$ and $Q$ are continuous. 
  
 Finally, for any sequence of non-negative numbers $h_{n}$ that converges to some $h$, the assumptions that $ \mathbb{E} (T(0)) > \mathbb{E} (S)$ and that $T(h)$ and $S$ have bounded variances, guarantee that the sequence of invariant distributions of the G/G/1 queue $\mu_{h_{n}}$ has bounded first two moments, and hence, it is tight. Thus, we can apply Proposition \ref{Prop:existence2} to conclude that an invariant distribution exists which completes the proof.  
\end{proof}

\begin{proof} [Proof of Corollary \ref{Coro:nonlinear}]
Existence follows 
immediately
 from Proposition \ref{Prop:existence}.  For uniqueness,  we need to show that the conditions of Theorem \ref{Theorem: unique} holds. We let $S=\{1,\ldots,n\}$ with the standard order, $H(\mu)=G \left (\mu(\{1\}),\ldots,\mu ( \{n \}) \right )$, and $D$ to be the set of increasing functions, so $\succeq_{D}$ is equivalent to $\succeq_{SD}$ and $(\mathcal{P}(S),\succeq_{D})$ is a complete lattice. Note that $H$ is increasing with respect to $\succeq_{SD}$ because $\mu \succeq_{SD} \mu' $ holds if and only if $\left (\mu(\{1\}),\ldots,\mu ( \{n \}) \right ) \geq_{m}\left (\mu'(\{1\}),\ldots,\mu ' ( \{n \}) \right ) $ and from the assumption that $G$ is increasing with respect to $\geq_{m}$. 
    
     Condition (1) implies that $Q$ is $D$-preserving, Condition (2) implies that $Q$ is $D$-decreasing, and Condition (3) implies that Property (U) holds.  Thus, we can apply Theorem \ref{Theorem: unique} to prove that $Q$ has at most one invariant distribution.  

    We can identify $Q$ with the stochastic matrix $P$ by $P_{ij} (\cdot) = Q(i,\cdot,\{ j \})$, and hence, using the definition of the invariant distribution, the Corollary follows from Theorem \ref{Theorem: unique}. 
\end{proof}

\begin{proof} [Proof of Corollary \ref{Coro:wealth}]
    For existence, continuity of $H$ and $Q$ follows immediately from the assumptions. Now note that the state space is bounded because the random variables $R_{i}$, $Y$, and the policy functions $g_{i}$ are bounded. In particular, we let the state space be the compact set $S=[0,nM\overline{r} + \overline{y}]$. Hence, we can use Proposition \ref{Prop:existence}  to conclude that $Q$ has an invariant distribution.
      
     For uniqueness, we need to show that the conditions of Theorem \ref{Theorem: unique} hold. We let $D$ to be the set of increasing functions, so $\succeq_{D}$ is equivalent to $\succeq_{SD}$. 
    
   It is immediate that Condition (2) implies that $Q$ is $D$-preserving and Conditions (2) and (3) imply that $Q$ is $D$-decreasing. Thus, we can apply Theorem \ref{Theorem: unique} to prove that $Q$ has at most one invariant distribution.
\end{proof}

\subsection{Proof of Claims 1,2,3,4,5}

\begin{proof}
    [Proof of Claim \ref{claim1}]
We let $D$ to be the set of all  increasing functions. Clearly $H$ is increasing with respect to $\succeq_{D}$ because $m$ is increasing. Property (C) holds for AR(1) process with $a \in (0,1)$, (see, for example, \cite{light2024note}).   Let $f: \mathbb{R} \rightarrow \mathbb{R}$ be increasing.  Then 
$$ \int f(y)Q(x, h, dy) = \int f(ax -  h + \epsilon ) \phi (d(\epsilon)) $$ 
is increasing in $x$ and decreasing in $h$ so $Q$ is $D$-preserving and $D$-decreasing. Hence, we can apply Theorem \ref{Theorem: unique} to conclude that $Q$ has at most one invariant distribution.

For existence, note that the quadratic growth condition and the fact that the variance of $\epsilon_{t}$ is finite imply that $H(\mu_{h})$ is finite for every $h$. Furthermore, if $h_{n}$ converges to $h$, then it follows that the sequence $\mu_{h_{n}}(dx)$ of invariant distributions of the AR(1) process given the parameter $h_{n}$ has bounded first two moments, and hence,  $\{\mu_{h_{n}} \}$ is a tight sequence of probability measures. In addition it is immediate that $Q$ is continuous and $H$ is continuous on $\{\mu_{h} : h\in [h',h''] \}$ as $\mu_{h}$ has bounded first two moments and $m$ is continuous with a quadratic bound. 

From the proof of Theorem \ref{Theorem: unique} we have that $H(\mu_{h})$ is decreasing in $h$. This implies that  we can find  $h',h'' \in \mathbb{R}$, $h''>h'$, such that $h'' \geq H(\mu_{h''})$ and $h' \leq H(\mu_{h'})$ (e.g., by letting $h' = -|c|-1$ and $h'' = |c|+1 $ where $H(\mu_{0}) = c$).

Thus, from Proposition \ref{Prop:existence2} existence follows. 
\end{proof}

\begin{proof}
    [Proof of Claim \ref{claim2}]
We let $D$ to be the set of all the functions that are increasing in the first argument. Clearly $H$ is increasing with respect to $\succeq_{D}$. We need to show that  $Q$ is $D$-preserving and $D$-decreasing in order to use Theorem \ref{Theorem: unique}. 
Let $f \in \mathbb{R}^{\mathbb{R}^{2} }$ be increasing in the first argument.  Then 
$$ \int f(y_{1},y_{2})Q((x_{1},x_{2}),h, dy) = \int f(ax_{1} -  h + \epsilon_{1},k(x_{2}) + \epsilon_{2}) \phi (d(\epsilon_{1},\epsilon_{2})) $$ 
is increasing in the first argument and decreasing in $h$ so $Q$ is $D$-preserving and $D$-decreasing. Hence, we can apply Theorem \ref{Theorem: unique} to conclude that $Q$ has at most one invariant distribution. Existence of an invariant distribution follows by the same argument as in Claim \ref{claim1}.
\end{proof}

\begin{proof}
    [Proof of Claim \ref{claim3}]
Consider the set of functions $D$ such that $f: \mathbb{R}^{n} \rightarrow \mathbb{R}$ is in $D$ if $f(x) = \sum _{i=1}^{n} y_{i}x_{i} + c$ for some $y \in O$ and $c \in \mathbb{R}$. Property (C) holds (see Example 1 in \cite{light2024note}).  It is immediate that $H$ is increasing with respect to $\succeq _{D}$. 
   
   We now show that $Q$ is $D$-preserving and $D$-decreasing. Let $f \in D$ so $f(x) = \sum _{i=1}^{n} y_{i}x_{i} + b$ for some $y \in O$.
   
   We have 
 \begin{align*}
     v(x):=\int f(x')Q(x,h,dx') & = \int f(a_{1}x_{1} - \beta_{1}h+\epsilon_{1},\ldots,a_{n}x_{n}-\beta_{n}h + \epsilon_{n}) \phi (d \epsilon ) \\
     & = \int \sum _{i=1}^{n} y_{i}(a_{i}x_{i} - \beta_{i}h + \epsilon_{i})\phi (d \epsilon ) + b \\
     & = \sum _{i=1}^{n} y_{i}'x_{i} + b'
      \end{align*}
with $y_{i}' = a_{i}y_{i}$ and $b' = \int \sum _{i=1}^{n} y_{i} ( -\beta_{i}h + \epsilon_{i})\phi (d \epsilon ) + b$. Note that $y' $ is in $O$ as $y \in O$ and $a_{i} \geq 0$ for all $i$. Hence, $v$ is in $D$ which means that $Q$ is $D$-preserving.

To show that $Q$ is $D$-decreasing  let $h_{2} \geq h_{1}$ and note that 
\begin{align*}
\int f(x')Q(x,h_{2},dx') & =  \int \sum _{i=1}^{n} y_{i}(a_{i}x_{i} - \beta_{i}h_{2} + \epsilon_{i})\phi (d \epsilon ) + b \\
     & \leq  \int \sum _{i=1}^{n} y_{i}(a_{i}x_{i} - \beta_{i}h_{1} + \epsilon_{i})\phi (d \epsilon ) + b  \\
     & = \int f(x')Q(x,h_{1},dx')
\end{align*}
where the inequality follows from the fact that $y$ and $\beta$ are in $O$ so $y_{i}\beta_{i} \geq 0$ for all $i$. Thus, $Q$ is $D$-decreasing.

To prove existence, note that we can find $H(\mu_{h})$ directly. A simple calculation shows that $H(\mu_{h}) = \sum _{i=1}^{n} \gamma_{i}(-h+e_{i})/(1-a_{i})$ where $e_{i}$ is the expected value of $\epsilon_{i}$. Thus, we can find  $h',h'' \in \mathbb{R}$, $h''>h'$, such that $h'' > H(\mu_{h''})$ and $h' < H(\mu_{h'})$. In addition, it is easy to see that the tightness condition of Proposition \ref{Prop:existence2} holds as the sequence $\{\mu_{h_{k}} \}$ has bounded first two moments whenever $h_{k}$ converges to some $h$.
\end{proof}

\begin{proof} [Proof of Claim \ref{claim:supermodular}]
    Fix the finite lattice \(S=\{0,1\}^{n}\) and let \(D\) be the set of all super‑modular functions on \(S\). 

Note that the nonlinear Markov kernel \(Q\) is defined for each \(y\in\{0,1\}^n\) by
\[
Q\bigl(x,h;\{y\}\bigr)
=\;\rho(x,h)\,\mathbf1_{\{y=(z,\dots,z)\}}\;p^z(1-p)^{1-z}
\;+\;(1-\rho(x,h))\;\prod_{i=1}^n p^{y_i}(1-p)^{1-y_i}.
\]

 Let \(f\in D\) and let $C_0=\mathbb{E}\bigl[f(Z_1,\ldots,Z_n)\bigr]$ and 
$C_1=\mathbb{E}\bigl[f(Z,\ldots,Z)\bigr].$  It can be shown that  \((Z,\ldots,Z)\) dominates \((Z_1,\ldots,Z_n)\) in super‑modular order (e.g., \cite{hu2005dependence}). Hence,  \(C_1\ge C_0\). 

For each \(x\in S\) and \(h\in[0,1]\),
we have 
$$ \sum _{y \in S} f(y) Q(x,h,\{y\})  
     =C_0+\bigl(C_1-C_0\bigr)\rho(x,h).
$$
which is supermodular in $x$ and decreasing in $h$ because the coefficient \(C_1-C_0\) is non‑negative and
\(\rho(\cdot,h)\) is super‑modular in $x$ and decreasing in $h$ by assumption. Hence
 \(Q\) is \(D\)-preserving and $D$-decreasing.

 To show that $H$ is increasing let \(\mu_2\succeq_{D}\mu_1\) and define 
$
f^*(y)=\mathbf{1}_{\{y=(1,\dots,1)\}}\,,
$
which is super‑modular on \(\{0,1\}^n\).  Then 
$$ H(\mu_{2}) = \mu_{2}\bigl(\{(1,\dots,1)\}\bigr) = \int f^*(y)\,\mu_{2}(dy) \geq  \int f^*\,d\mu_1 (dy)
=\mu _{1}\bigl(\{(1,\dots,1)\}\bigr) = H(\mu_{1}),
$$ 
and $H$ is increasing with respect to $\succeq_{D}$. 

In addition, for each \(h\) the linear Markov chain \(Q(\cdot,h,\cdot)\) assigns strictly positive
probability to every state of \(S\)
because \(0<p<1\) and \(1-\rho(x,h)>0\) 
Hence the chain is irreducible and aperiodic on a finite state space and Property (C) holds. 

Thus, from Theorem \ref{Theorem: unique} the nonlinear Markov chain $Q$ has at most one invariant distribution.  Existence readily follows from using the continuity of $\rho$ and applying Proposition \ref{Prop:existence}. This proves Claim \ref{claim:supermodular}.

Now assume that \(D_{\uparrow}\) is the set of all increasing functions on \(S=\{0,1\}^n\) and assume for simplicity $n=2$. Let 
$k(y) \;=\; y_1 + y_2 - y_1y_2,
$
which is increasing in each coordinate. We have 
\[
\sum _{y} k(y) Q(x,h, \{y \}) = Q(x,h,\{1,0\}) + Q(x,h,\{1,1\}) + Q(x,h,\{0,1\}) = 2p - p^2 \;-\;(p-p^2)\,\rho(x,h),
\]
which is not necessarily increasing. Hence, \(Q\) fails to be \(D_{\uparrow}\)\!--\,preserving. 
\end{proof}

\begin{proof}
    [Proof of Claim \ref{claim_queue}]
Let 
\begin{equation} \label{Eq:x} h = \frac{\mathbb{E}(S^{2})} {\sqrt{\mathbb{E}(S)^{2}+2\mathbb{E}(S^{2})} - \mathbb{E}(S)} \end{equation}

and consider the linear Markov chain $W_{t+1} = \max (0, W_{t} + S_{t} - T_{t}(h) )$. Then it has a unique invariant distribution if $\mathbb{E}T_{t}(h) = h > \mathbb{E}(S) $ (see Theorem 19.3.5 in \cite{meyn2012markov}) which holds because
$$ \mathbb{E}(S)\sqrt{\mathbb{E}(S)^{2}+2\mathbb{E}(S^{2})} =\sqrt{\mathbb{E}(S)^{4}+2\mathbb{E}(S^{2})\mathbb{E}(S)^{2}}  < \sqrt{\left ( \mathbb{E}(S)^{2} + \mathbb{E}(S^{2}) \right ) ^{2}} =  \mathbb{E}(S)^{2} + \mathbb{E}(S^{2})  $$
which implies that $h > \mathbb{E}(S)$. 
Let $W_{\infty}$ be the random variable with the law $\mu^{*}$ where $\mu^{*}$ is unique invariant distribution of  the linear Markov chain $(W_{t})_{t \in \mathbb{N}}$.

From the Pollaczek-Khinchin formula (see Equation (8.1) in Chapter 8 in \cite{cooperintroduction}) the stationary expected waiting time  is given by $\mathbb{E}(W_{\infty }) = \lambda (h) \mathbb{E}(S^{2})/(2(1-\lambda(h)\mathbb{E}(S)))$. Using the fact that $\lambda (h) = 1/h$, and algebraic manipulations, we see that $h = \mathbb{E}(W_{\infty })$. Hence, $\mu^{*}$ is an invariant distribution of the nonlinear Markov chain given in Equation (\ref{Eq:queue}). Uniqueness follows from Corollary \ref{Corr:Queue}.  

For $M/M/1$ queue $S$ is an exponential random variable with a parameter $\mu$, so $\mathbb{E}(S) = 1/\mu$ and $\mathbb{E}(S^{2}) = 2/\mu^{2}$ and we get $$\mathbb{E} (W_{ \infty})  = \frac{ 2}  {(\sqrt{5}-1)\mu} = \frac{ 2\mathbb{E}(S)} {\sqrt{5}-1}$$
which completes the proof. 
\end{proof}

\subsection{Uniqueness via the Convex Stochastic Order} \label{Sec:convexOrder}

In this section, we expand Section \ref{sec:flexibility} by presenting an example where the convex stochastic order is used to establish the uniqueness of an invariant distribution and the standard first order stochastic dominance would not satisfy the required conditions for uniqueness. 

\begin{example}[Convex Stochastic Order]
\label{ex:convex_stochastic}
Suppose that the state space is $S = \mathbb{R}$ and 
$a\in(-1,1)$.  
Let $(\epsilon_t)_{t\ge1}$ be i.i.d. with mean zero, finite moments and law  $\phi$.   Let $\sigma(h)$ be a positive, continuous and decreasing function  
and consider 
the nonlinear Markov chain 
\begin{equation*}
  X_{t+1}  \;=\;a\,X_t\;+\;\sigma\!\bigl(H(\mu_t)\bigr)\,\epsilon_{t+1},
\end{equation*}
where $H(\mu) = \int m(x) \mu(dx) $ for some continuous and convex function $m$ on $\mathbb{R}$ such that $|m(x)| \leq C_{0} + C_{1}x^{k}$ for some $k \geq 1$ and constants $C_{0}$, $C_{1}$. Let $D$ be the set of all convex functions on $S$ so $\succeq_D$ is the convex stochastic order. 

Clearly $H$ is increasing with respect to 
$\succeq_D$.  Property (C) follows from standard arguments as in Claim \ref{claim1}. 

Let $f: \mathbb{R} \rightarrow \mathbb{R}$ be convex. Then the function
$$ v(x,h):= \int f(y)Q(x, h, dy) = \mathbb{E} f(ax +  \sigma(h)\epsilon)  $$ 
where the expectation is with respect to $\epsilon$, is convex in $x$ as a composition of a convex and a linear function. Hence,  $Q$ is $D$-preserving.

Now let $h_{2} \geq h_{1}$ and $Y = \sigma(h_{2}) \epsilon$, $c= \sigma (h_{1}) / \sigma(h_{2})$. Note that $c \geq 1$ because $\sigma$ is decreasing. 

In addition, from Jensen's inequality we have
$$ \mathbb{E} f(ax + Y) \leq  \frac{1}{c} \mathbb{E} f(ax+cY) + \left ( 1 - \frac{1}{c} \right )f(ax). $$
Hence, using Jensen's inequality again we have  
$$\mathbb{E} f(ax+Y)= c\mathbb{E} f(ax + Y) -  \left ( c - 1 \right ) \mathbb{E} f(ax + Y) \leq  c \mathbb{E} f(ax + Y) -  \left ( c - 1 \right )f(ax)  \leq   \mathbb{E} f(ax+cY) . $$
That is, 
$$v(x,h_{2}) = \mathbb{E} f(ax+\sigma (h_{2})\epsilon) \leq \mathbb{E} f(ax+\sigma (h_{1})\epsilon) = v(x,h_{1})$$
so $D$-decreasing. Hence, we can apply Theorem \ref{Theorem: unique} to conclude that $Q$ has at most one invariant distribution.

The proof of existence follows from similar arguments to the proof in Claim \ref{claim1} so it is omitted. We conclude that $Q$ has a unique invariant distribution.

\end{example}

\subsection{On non-contraction and numerical examples} \label{sec:Non-contraction}

In this section we show that the nonlinear operator $T$ is generally not a contraction in our examples and applications.
 We focus on three  tractable cases: Example~\ref{example:flexible} (a one–dimensional nonlinear autoregression), the strategic M/M/1 queue  (Section~\ref{sec:queueing}), and the nonlinear fixed‑point equations studied in Section~\ref{Sec:nonlinear}.  In each instance, we show that the associated dynamics are generally not contractive even in simple parameter configurations.  We note that the other two applications (inventory dynamics  and wealth distributions) have transitions that depend on policy functions that typically come from a dynamic programming problem and do not have closed‑form expressions so  establishing contraction there would be much more involved.  In Section \ref{sec:noncontraction:wealth} we derive such a policy function numerically in the wealth distribution application and compute the corresponding invariant wealth distribution using Algorithm \ref{alg:bisection_simplified}.

We first recall the definitions of the Wasserstein and total variation metrics between two probability measures 
$\mu$ and $\nu$ on a metric space $(\mathcal{S}, d)$ that is  a Polish space.

For $p \ge 1$, the \emph{Wasserstein distance of order $p$} is defined as
\begin{equation*}
    W_p(\mu, \nu) :=  \inf_{\pi \in \Pi(\mu, \nu)} \left( \int_{\mathcal{S} \times \mathcal{S}} d(x, y)^p \, d\pi(x, y) \right)^{1/p},
\end{equation*}
where $\Pi(\mu, \nu)$ denotes the set of all couplings of $\mu$ and $\nu$. If $\mu$ and $\nu$ are probability measures on $\mathbb{R}$ with cumulative distribution functions (CDFs) \( F_1 \) and \( F_2 \), then the Wasserstein distance of order \( p \ge 1 \) is given by:
\begin{equation} \label{Eq:Wasserstein_R}
    W_p(\mu, \nu) = \left( \int_0^1 \left| F_1^{-1}(q) - F_2^{-1}(q) \right|^p \, dq \right)^{1/p},
\end{equation}
where \( F_1^{-1} \) and \( F_2^{-1} \) are the quantile functions (inverse CDFs).

The \emph{total variation distance} between $\mu$ and $\nu$ is defined as
\begin{equation*}
    \|\mu - \nu\|_{TV} := \sup_{A \in \mathcal{B}(\mathcal{S})} |\mu(A) - \nu(A)|. 
\end{equation*}
If $\mu$ and $\nu$ admit densities $f$ and $g$ with respect to a common reference measure, then the total variation distance simplifies to
\begin{equation} \label{Eq:totalvariationPDF}
    \|\mu - \nu\|_{TV} = \frac{1}{2} \int_{\mathcal{S}} |f(x) - g(x)| \, dx.
\end{equation}

\subsubsection{Contraction in Example \ref{example:flexible}}

Consider the nonlinear Markov chain in Example \ref{example:flexible} given in Equation \ref{Eq:ex_1(i)} with $m(x) = \beta x$ for some $\beta > 0$. As a special case of Claim \ref{claim1}, this nonlinear Markov chain has a unique invariant distribution using the monotonicity arguments developed in this paper. We now show it is generally not a contraction in the Wasserstein distance and the total variation distance. We then show that  augmenting the nonlinear Markov chain with an additional variable to construct a linear Markov chain would not generally be a useful approach to prove uniqueness of an invariant distribution. 

\textbf{Wasserstein Distance.} 
Let's assume $ p \geq 1$ and the $p$th moment of $\epsilon$ exists. 

Consider two simple Dirac measures $ \mu_1 = \delta_h$ and $\mu_2 = \delta_{h+\delta}$ for some $h \in \mathbb{R}$ and $\delta > 0$. Then $W_{p}(\mu_{1},\mu_{2}) = \delta$. 

For $\mu_1 = \delta_h$, we have $x=h$ and $H(\mu_1)=\beta h$ so $T\mu_1 = \operatorname{Law}(\epsilon + (a-\beta)h)$. Similarly, for $\mu_2 = \delta_{h+\delta}$, we have $x=h+\delta$ and $H(\mu_2)=\beta (h+\delta)$ so  $T\mu_{2} = \operatorname{Law}(\epsilon + (a-\beta)h + (a-\beta)\delta)$.  Thus, $T\mu_2$ is a simple translation of $T\mu_1$ by a constant shift of $(a-\beta)\delta$. The Wasserstein distance between such laws $W_{p}(T\mu_{1},T\mu_{2})$ is $|a - \beta| \delta$. 

Hence, we conclude that a necessary condition for $T$ to be a contraction is $|a-\beta| < 1$ which is not generally the case (e.g., if $\beta \geq 2$).  On the other hand, a sufficient condition for contraction in \( W_1 \) is the relatively strong requirement \( |a| + |\beta| < 1 \).\footnote{Indeed, let $\pi$ be an optimal coupling of $\mu$ and $\nu$, so $(X,Y) \sim \pi$ satisfies
$
W_{1}(\mu,\nu) \;=\; \mathbb{E}_{\pi} \bigl[ |X - Y| \bigr].
$
Let $\epsilon \sim \phi$ be an independent noise term.  
Define the valid coupling of $T\mu$ and $T\nu$, 
$
(X',Y') \;=(aX - \beta \mathbb{E}[X] + \epsilon, a Y - \beta \mathbb{E}[Y] + \epsilon).
$
Therefore by using the triangle inequality and Jensen's inequality we have
\[
W_{1}(T\mu,T\nu) \leq \mathbb{E}_{\pi} \bigl[ |a(X-Y) - \beta \mathbb{E}_{\pi}[X-Y]| \bigr]
\;\le\;
|a|\,\mathbb{E}_{\pi} \bigl[ |X-Y| \bigr] + |\beta|\,|\mathbb{E}_{\pi}[X-Y]| \leq \bigl(|a| + |\beta|\bigr) \, \mathbb{E}_{\pi} \bigl[ |X-Y| \bigr].
\]
That is, 
\[
W_{1}(T\mu,T\nu) \;\le\; \bigl(|a| + |\beta|\bigr) \, W_{1}(\mu,\nu).
\]} 
 Intuitively, the nonlinear mapping $T$ amplifies differences in the mean when $\beta$ is large: a small change in the current distribution’s mean leads to an even larger shift in the next period's mean.  This is in addition to the standard auto regressive feedback captured by $a$. In more complex applications we consider next where the entire distribution determines the dynamics, contraction is typically far more difficult to satisfy. 
 
\textbf{Total Variation Distance.}
Let $\mu_1=\mathcal{N}(m_1,\sigma^2)$ and $\mu_2=\mathcal{N}(m_2,\sigma^2)$ where $\mathcal{N}(\mu_{i},\sigma^{2})$ is the normal random variable with mean $\mu_{i}$ and variance $\sigma ^{2}$ and probability density function $ \frac{1}{\sqrt{2\pi\sigma^2}} \exp\left( -\frac{(x - m_i)^2}{2\sigma^2} \right)$.

A simple calculation using Equation (\ref{Eq:totalvariationPDF}) shows that the total variation distance is given by
\begin{equation} \label{Eq:TotalVarNormal}
   d_{TV}(\mu_1,\mu_2)=
   2\Phi\!\Bigl(\frac{|m_2-m_1|}{2\sigma}\Bigr)-1,
\end{equation}
where $\Phi$ is the standard normal CDF. 

Consider two normal random variables $\mu=\mathcal{N}(0,1)$ and $\nu= \mathcal{N}(1,1)$ as an example and assume that $\epsilon = \mathcal{N}(0,1)$.  Then using Equation (\ref{Eq:TotalVarNormal}) we have   $d_{TV}(\mu,\nu)=2\Phi\!\Bigl(\tfrac12\Bigr)-1$.

Now note that  $T\mu=\operatorname{Law}\bigl(aX+\varepsilon\bigr)=\mathcal{N}\!\Bigl(0,\;a^{2}+1\Bigr)$ 
           where we used the fact that  the sum of two independent normally distributed random variables is normal and its mean is the sum of the two means, and its variance is the sum of the two variances. Similarly, $T\nu = \mathcal{N}(a-\beta , a^{2}+1)$.  

Using again Equation (\ref{Eq:TotalVarNormal}) we have  $d_{TV}(T\mu,T\nu)= 2 \Phi \!\Bigl(\tfrac {|a-\beta|} {2 \sqrt {a^{2} + 1}}\Bigr) - 1$. 

Hence, $d_{TV}(T\mu,T\nu) \geq d_{TV}(\mu,\nu)$ if $|a-\beta| \geq \sqrt {a^{2} + 1}$ and $T$ is generally not a contraction.

\textbf{Augmented Markov Chain.} Another possible approach is to augment the original nonlinear Markov chain to obtain a linear Markov process, thereby enabling the application of standard tools to establish existence and uniqueness of a stationary distribution \citep{meyn2012markov}. In the special case where the aggregator is given by \( m(x) = \beta x \), this augmentation can be achieved by introducing a single additional variable representing the mean, rather than augmenting with the entire distribution.

To do that we introduce the mean coordinate  as an additional variable: 
$m_t:=\mathbb{E}[X_t]$, and consider the augmented Markov chain $Z_t:=(X_t,m_t)\in\mathbb R^2 $. 
We start from any distribution of $X_0$ with finite first two moments and $m_0=\mathbb{E}[X_0]$. 

Taking expectations in the original nonlinear Markov chain given in Equation (\ref{Eq:ex_1(i)}) yields 
\begin{equation*}
  m_{t+1}=a\,m_t-\beta\,m_t+e=(a-\beta)m_t+e .
\end{equation*}
where $ e = \mathbb{E}(\epsilon_{t})$. 
Hence with 
\[
  A:=\begin{pmatrix}a&-\beta\\0&a-\beta\end{pmatrix},
  \quad
\xi_{t+1}:=\begin{pmatrix}\epsilon_{t+1}\\ e\end{pmatrix},
\]
the augmented Markov chain is an AR(1) process
\begin{equation}\label{eq:AR2}
  Z_{t+1}=A Z_t+\xi_{t+1},\qquad t\ge0 .
\end{equation}
on $\mathbb R^2$. 

Assume that $|a-\beta| > 1$.
Now it is easy to see that an invariant distribution for the augmented AR(1) process has to be of the form of the product measure $\mu^{*} \otimes\delta_{m^\star}$  where $  m^\star =\dfrac{e}{1-(a-\beta)}$. In addition, $\mu^{*} \otimes\delta_{m^\star}$ is an invariant distribution for the augmented AR(1) process if and only if $\mu^{*}$ is an invariant distribution of the original nonlinear Markov chain with $\mathbb{E}(X_{\infty}) = m^{*}$.  

However, when $|a - \beta| > 1$, the mean component $m_t$ diverges for almost every initial condition $m_0$. Consequently, classical techniques from the theory of linear Markov chains, such as drift conditions \cite{meyn2012markov}, that imply global stability do not apply.

\begin{remark}
We showed that the operator $T$ for the nonlinear Markov chain in Example~\ref{example:flexible}, given in Equation~\ref{Eq:ex_1(i)} with $m(x)=\beta x$ and studied above, is Lipschitz with respect to $W_{1}$. One may therefore wonder whether there is any necessary connection between Lipschitz continuity of $T$ with respect to the Wasserstein distance $W_p$ and the existence and uniqueness results established via our monotonicity approach. We now demonstrate that the answer is generally no. We present a cubic aggregator for which the associated nonlinear Markov operator $T$ is not globally Lipschitz with respect to $W_p$ for any $p\ge1$, yet it still possesses a unique invariant distribution, a result that can be established directly using our uniqueness theorem.

Let $S=\mathbb{R}$ be endowed with the usual order. Fix $a\in(0,1)$ and $\beta>0$, and let
$(\varepsilon_t)_{t\ge1}$ be i.i.d.\ with
$
\mathbb{E}[\varepsilon_1]=0$, $
\mathbb{E}[\varepsilon_1^2]=\sigma^2\in(0,\infty)$, $
\mathbb{E}\bigl[|\varepsilon_1|^3\bigr]<\infty$. 
For $r\ge1$ set
$
\mathcal{P}_r(\mathbb{R})
:=\Bigl\{\mu\in\mathcal{P}(\mathbb{R}) : \int_{\mathbb{R}} |x|^r\,\mu(dx)<\infty\Bigr\},
$
and define the cubic aggregator $H_{3}:\mathcal{P}_3(\mathbb{R})\to\mathbb{R}$ by
$
H_{3}(\mu):=\int_{\mathbb{R}} x^3\,\mu(dx).
$
Consider the nonlinear Markov chain
\[
X_{t+1}=aX_t-\beta\,H_{3}(\mu_t)+\varepsilon_{t+1},\qquad \mu_t:=\mathrm{Law}(X_t),
\]
and the associated one-step operator $T$ given by $T\mu \;:=\; \mathrm{Law}\bigl(aX-\beta H_{3}(\mu)+\varepsilon\bigr)$. 

Fix $p\ge1$ such that $\mathbb{E}|\varepsilon_1|^p<\infty$ and write $r:=\max\{3,p\}$.
Then $T(\mathcal{P}_r(\mathbb{R}))\subseteq\mathcal{P}_r(\mathbb{R})$, but 
$T:(\mathcal{P}_r(\mathbb{R}),W_p)\to(\mathcal{P}_r(\mathbb{R}),W_p)$ is \emph{not} globally Lipschitz. Indeed, for $x\neq y$ set $\mu=\delta_x$ and $\nu=\delta_y$.
Then $W_p(\mu,\nu)=|x-y|$ and $H_{3}(\delta_z)=z^3$. 

Moreover,
$
T\delta_z=\mathrm{Law}\bigl(az-\beta z^3+\varepsilon\bigr),
$ 
which is a translate of $\mathrm{Law}(\varepsilon)$ by the constant $az-\beta z^3$. Hence, 
\[
W_p(T\delta_x,T\delta_y)
=\big|a(x-y)-\beta(x^3-y^3)\big|
=\Big|\,a-\beta\,\frac{x^3-y^3}{x-y}\,\Big|\,|x-y|.
\]
Therefore,
\[
\frac{W_p(T\delta_x,T\delta_y)}{W_p(\delta_x,\delta_y)}
=\big|\,a-\beta(x^2+xy+y^2)\,\big|.
\]
Since $x^2+xy+y^2$ is unbounded on $\mathbb{R}^2$, no finite global $W_p$-Lipschitz constant can exist.

On the other hand, on the domain $W:=\mathcal{P}_3(\mathbb{R})$ and with $D$ the set of increasing functions, the monotonicity conditions for uniqueness follow from the same arguments as in Claim~\ref{claim1}. Hence Proposition \ref{Prop:local}  yields uniqueness, and
Proposition \ref{Prop:existence2} yields existence of an invariant distribution in $W$. That is, $T$ admits a unique invariant
distribution in $\mathcal{P}_3(\mathbb{R})$.
   
\end{remark}

\subsubsection{Contraction in Strategic Queueing }

Now consider the strategic queueing system we studied in Section \ref{sec:queueing}. Suppose for simplicity that  the system behaves as a strategic M/M/1 queue where  $S\sim\operatorname{Exp}(\mu)$ and $T(h)\sim\operatorname{Exp}(\lambda(h))$ with $\mu=1$.

We first define the CDF $F_{h;x}(t) := P\{Y_h(x) \le t\}$ of the random variable $Y_h(x) = \max\{0, x+S-T(h)\}$ which is given by $0$ for $t<0$ and
\begin{equation}
    \label{equation:F_laplace}
  F_{h;x}(t) = \begin{cases} 
    \displaystyle
      \frac{1}{1+\lambda(h)}\,\exp(\lambda(h)\,(t-x)),
      &0\le t<x,\\[8pt]
    \displaystyle
      1-\frac{\lambda(h)}{1+\lambda(h)}\,\exp(-(t-x)),
      &t\ge x.
  \end{cases}
\end{equation}

Indeed, this follows from the well known fact that the difference between the two exponential distributions $S$ and $T(h)$ with parameters $1$ and $\lambda(h)$, respectively is a Laplace distribution.

\textbf{Wasserstein Distance.} 
To test contraction we let 
\(\mu=\delta_{a}\) for $a > 1$ and perturb it to
\(\nu_{\varepsilon}=(1-\varepsilon)\,\delta_{a}+\varepsilon\,\delta_{c}\)
for some $c>a$ and small \(\varepsilon\). We will focus on $W_{1}$ for this example but similar computations to show non-contraction of $T$ can be done to other $p \geq 1$. 
The 1-Wasserstein distance between the measures is 
$
   W_{1}(\mu,\nu_{\varepsilon})
   \;=\;\varepsilon\lvert c-a\rvert.
$

We now compute the law of $T\mu$ and $T\nu_{\varepsilon}$. 
For \(\mu=\delta_{a}\) the mean is $a$ so the  
cumulative distribution function is given by $ F_{T\mu}(t) = F_{a;a}(t)$ (see Equation \ref{equation:F_laplace}). For  \(\nu_{\varepsilon}=(1-\varepsilon)\,\delta_{a}+\varepsilon\,\delta_{c}\) the mean is $\bar h=(1-\varepsilon)a+\varepsilon c$ so 
\[
   F_{T\nu_{\varepsilon}}(t)
     =(1-\varepsilon)\,F_{\bar h ; a}(t)
      +\varepsilon\,F_{\bar h ; c}(t),
   \quad
   \bar h=(1-\varepsilon)a+\varepsilon c. 
\]

The 1-Wasserstein distance between these probability measures is 
\[
   W_{1}(T\mu,T\nu_{\varepsilon})
   =\int_{0}^{\infty}
      \bigl|F_{T\mu}(t)-F_{T\nu_{\varepsilon}}(t)\bigr|\,dt,
\]
which can be evaluated analytically and numerically. In Figure \ref{fig:subW1} 
we evaluate  for small values of $\varepsilon$ 
and plot the ratio
\[
   K(\varepsilon)
   =\frac{W_{1}(T\mu,T\nu_{\varepsilon})}
          {W_{1}(\mu,\nu_{\varepsilon})}.
\]
and show that it is above $1$, i.e., $T$ is not a global contraction in the 
\(W_{1}\) metric.

\textbf{Total Variation.}  As in the Wasserstein metric example, to test contraction we let 
\(\mu=\delta_{a}\) for $a > 1$ and perturb it to
\(\nu_{\varepsilon}=(1-\varepsilon)\,\delta_{a}+\varepsilon\,\delta_{c}\)
for some $c>a$ and small \(\varepsilon\).  The TV distance is $\|\mu - \nu_\varepsilon\|_{TV} = \varepsilon$. We will consider the measurable set $A=\{0\}$. From Equation (\ref{equation:F_laplace}) and the analysis of the Wasserstein metric we have 
\begin{align*}
    T\mu(\{0\}) &= F_{a;a}(0) \\
    T\nu_\varepsilon(\{0\}) &= (1-\varepsilon)F_{\bar{h}(\varepsilon) ; a }(0)  + \varepsilon F_{\bar{h}(\varepsilon) ; c} (0). 
\end{align*}
where $\bar{h}(\varepsilon) := (1-\varepsilon)a + \varepsilon c$. 

We define the difference $D(\varepsilon) := T\mu(\{0\}) - T\nu_\varepsilon(\{0\})$ and note that 
$$ \frac{\|T\mu - T\nu_\varepsilon\|_{TV}}{\|\mu - \nu_\varepsilon\|_{TV}} \ge \frac{|D(\varepsilon)|}{\varepsilon} $$
from the definition of the TV metric. 

In Figure \ref{fig:subTV} we evaluate for small values of $\varepsilon$  and compute numerically $|D(\epsilon)|/\epsilon$ and show that it is above $1$, i.e., $T$ is not a global contraction in the TV metric.   

\begin{figure}[htbp]
  \centering
  \begin{subfigure}[b]{0.45\textwidth}
    \centering
\includegraphics[width=\textwidth]{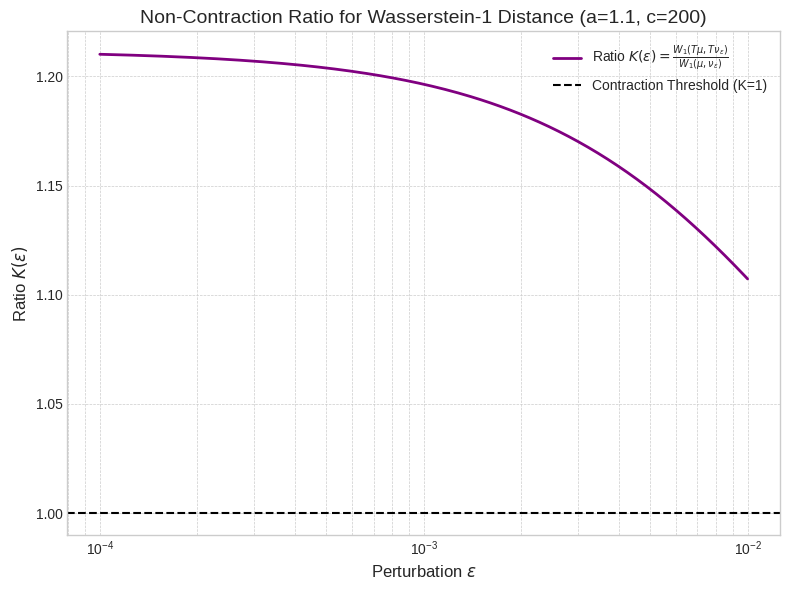}
    \caption{Non-contraction in the Wasserstein metric}
    \label{fig:subW1}
  \end{subfigure}
  \hfill
  \begin{subfigure}[b]{0.45\textwidth}
    \centering
\includegraphics[width=\textwidth]{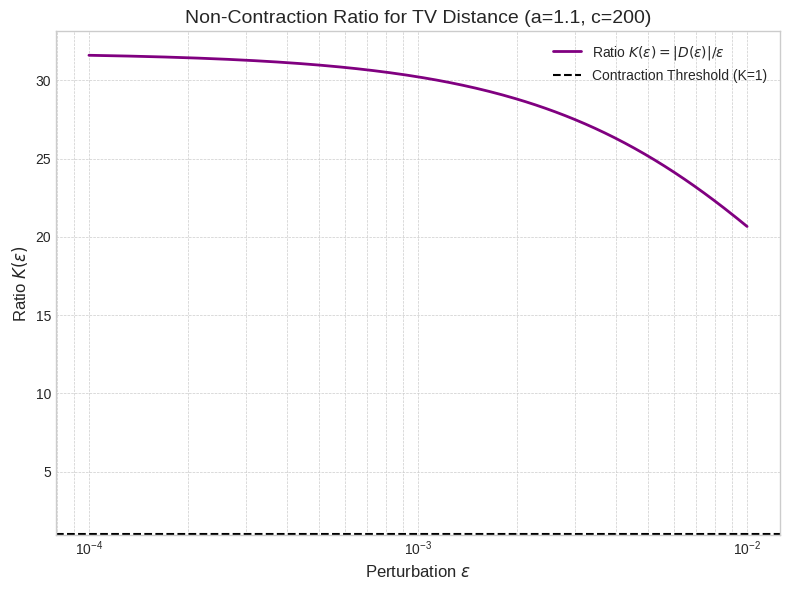}
    \caption{Non-contraction in the TV metric}
    \label{fig:subTV}
  \end{subfigure}
  \caption{The Figures show the non-contraction in the Wasserstein and TV metric for the M/M/1 strategic queueing systems with parameters $\lambda(h) = 1/(h+1)$, $a=1.1$, $c=200$. }
  \label{fig:overall}
\end{figure}

Despite this non-contraction, Algorithm \ref{alg:bisection_simplified} computes the equilibrium very fast and with only 15 iterations for the bisection method as we see in Figure \ref{Figure:BisectionQueue}.

\begin{figure}[htbp]
  \centering
  \begin{subfigure}[b]{0.45\textwidth}
    \centering
\includegraphics[width=\textwidth]{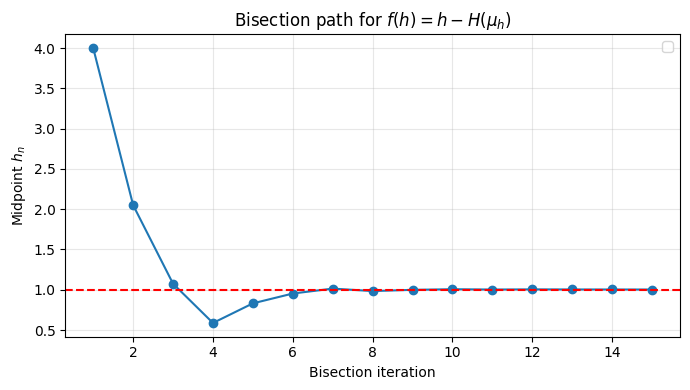}
    \caption{Algorithm \ref{alg:bisection_simplified} convergence}
   \label{Figure:BisectionQueue}
  \end{subfigure}
  \hfill
  \begin{subfigure}[b]{0.45\textwidth}
    \centering
\includegraphics[width=\textwidth]{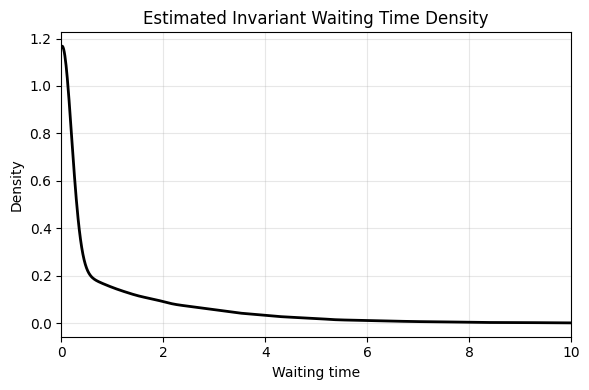}
    \caption{Equilibrium waiting time distribution}
    \label{Figire:DensityQueue}
  \end{subfigure}
  \caption{ Figure \ref{Figure:BisectionQueue} shows the sequence of midpoints
  $\{h_n\}$ generated by 
(Algorithm~\ref{alg:bisection_simplified}) when applied to the strategic M/M/1 queue 
           with
           $S \sim \operatorname{Exp}(1)$ and arrival rate
    $\lambda(h)=1/(1+h)$.
           At each node the stationary mean $H(\mu_h)=\mathbb E[X_h]$
           is estimated by a Monte Carlo simulation of the
           linear $M/M/1$ queue. Algorithm \ref{alg:bisection_simplified} finds 
          the unique fixed point
           $h^\ast \approx 1.00$ (which can be calculated explicitly in this simple setting as in Claim \ref{claim_queue}) in fewer than 15 iterations with tolerance level $10^{-4}$. Figure \ref{Figire:DensityQueue} shows  the estimated invariant density of waiting times, obtained by applying a Gaussian kernel density estimator. }
  \label{fig:QueueOverall}
\end{figure}

\subsubsection{Contraction in the Wealth–Distribution Model}
\label{sec:noncontraction:wealth}

In this section we consider the wealth distribution
application presented in  Section~\ref{sec:wealth}. In particular, we consider a single asset case with a fixed interest rate that is determined by the aggregate agents' behavior in the economy.\footnote{Specifically, we consider the Aiyagari model \citep{aiyagari1994} as presented in \cite{light2020uniqueness} (see the full details there). The production function is given by 
$f(k)=k^{\alpha}$ with $\alpha=0.5$ and the interest rate is therefore given by 
$R(H(\mu))=\alpha H(\mu)^{\alpha-1}-\delta+1$ with $\delta=0.1$ (as in \cite{light2020uniqueness}) where $H(\mu) = \int x \mu(dx)$.   The nonlinear Markov chain is then given by 
\(X_{t+1}=R(H(\mu_t))\,g\bigl(X_t,R(H(\mu_t))\bigr)+Y_{t+1}\),
where $\{Y_{t}\}$ are i.i.d. labor income shocks and \(g(\,\cdot\,;R)\) is the optimal saving policy that is determined by an income–fluctuation problem with logarithmic
utility. In this setting it can be shown that the savings policy function is increasing in the interest rate and current wealth and Property (C) holds \citep{light2020uniqueness} so we can use Corollary \ref{Coro:wealth} to show that the nonlinear Markov chain has indeed a unique invariant distribution. }
In this model the optimal savings policy does not have closed-form solution. We compute it by using the value function iteration algorithm. In Figure \ref{fig:WealthOverall} we first test numerically if $T$ is a global contraction in the \(W_1\) metric.\footnote{In particular, given a fixed rate $R$, we solve the Bellman equation for the income fluctuation problem 
\[
   V(x)=\max_{a\in[0,x]}
     \Bigl\{\log(x-a)+\beta \mathbb{E}[V(Ra+Y')]\Bigr\},
\]
on the cash–on–hand grid \(x\in[10^{-3},15]\) using value–function
iteration with $800$ iterations and a $25$‑point grid for the savings
between $0$ and $x$ and interpolate the optimal savings to derive the optimal policy function $g(x,R)$.  
Then, to implement $T$, for any empirical wealth distribution represented by a sample
\(\{x_i\}_{i=1}^N\) we set \(H(\mu)=\frac1N\sum_i x_i\),
compute \(R=R(H(\mu))\) and the corresponding policy $g(\cdot, R)$, draw i.i.d.\ income $Y_i\in\{1,3\}$ with equal
probabilities for each agent, and return
\(x_i' = R\,g_R(x_i)+Y_i\).
} We then show that Algorithm \ref{alg:bisection_simplified} converges fast to the unique invariant distribution of this model and plot the corresponding wealth distribution.

\begin{figure}[htbp]
  \centering
  \begin{subfigure}[b]{0.5\textwidth}
    \centering
\includegraphics[width=\textwidth]{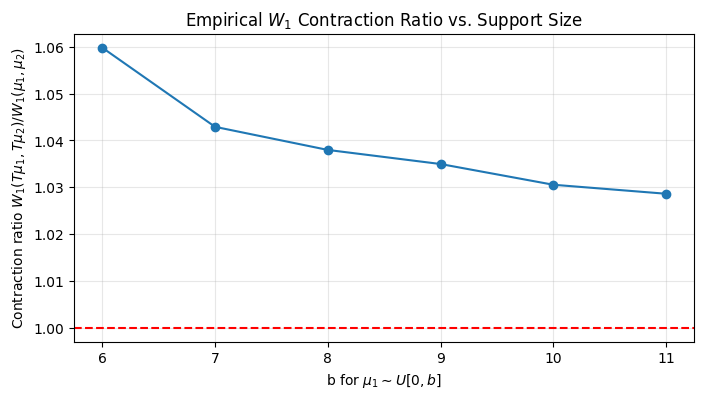}
    \caption{Non-contraction in the wealth distribution model}
   \label{Figure:ContractionWealth}
  \end{subfigure}
  \hfill
  \begin{subfigure}[b]{0.75\textwidth}
    \centering
\includegraphics[width=\textwidth]{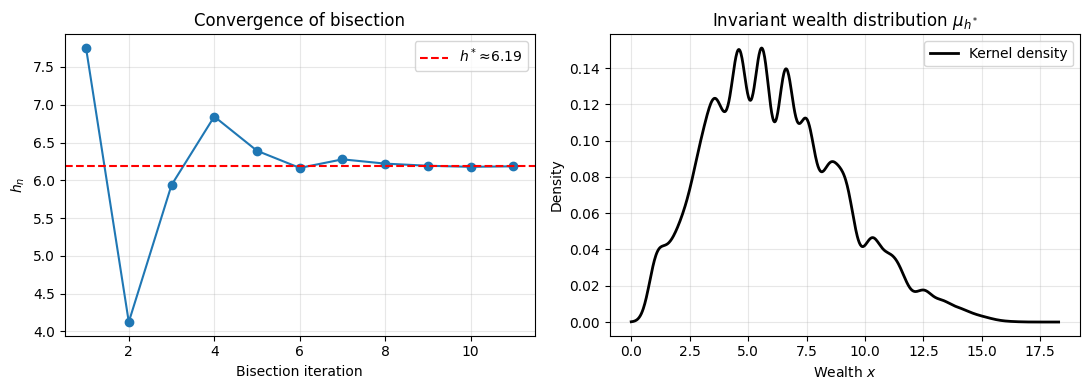}
    \caption{Computation of the unique invariant wealth distribution}
    \label{Figure:DensityWealth}
  \end{subfigure}
  \caption{In Figure \ref{Figure:ContractionWealth} 
we compare two initial laws:
 \(\mu_1\sim U[0,b]\) with $b\in\{6,7,\dots,11\}$,
  and \(\mu_2\sim U[0,5]\).
With $N=4\times10^{5}$ draws from these distributions we compute
\(W_1(\mu_1,\mu_2)\) and \(W_1(T\mu_1,T\mu_2)\), then plot the ratio
$
   \rho(b)=\frac{W_1(T\mu_1,T\mu_2)}{W_1(\mu_1,\mu_2)}$
against $b$.
We see that \(\rho(b)>1\); so \(T\) is not  a global
contraction. In Figure \ref{Figure:DensityWealth} we apply Algorithm \ref{alg:bisection_simplified}
 to compute the invariant distribution. 
The left panel displays the midpoint sequence $\{h_n\}$ generated by the
bisection and convergence to the fixed point
$h^*$ is achieved in only 11 iterations. 
The right panel plots the density of the invariant wealth distribution 
$\mu_{h^*}$ 
where we apply a
Gaussian kernel density estimator to plot it.}
  \label{fig:WealthOverall}
\end{figure}

While the monotonicity conditions of Theorem~\ref{Theorem: unique} guarantee uniqueness of the invariant distribution, the numerical simulations suggest that standard contraction arguments are inapplicable, as contraction is again not satisfied in this application.

\newpage

\bibliographystyle{ecta}
\bibliography{unique}

\end{document}